\documentclass[preprint,10pt]{elsarticle}

\usepackage{bm}
\usepackage{amsfonts}
\usepackage{amscd}
\usepackage{latexsym}
\usepackage{pifont}
\usepackage{fleqn}
\usepackage{pifont}
\usepackage{natbib}
\usepackage{subfigure}
\usepackage[margin=5pt,font=small,labelfont=bf]{caption}
\usepackage{epsf}
\usepackage{indentfirst}
\usepackage{mathrsfs}
\usepackage{amsmath}
\usepackage{color,xcolor}
\usepackage{tikz}
\usepackage{float}
\usepackage{geometry}
\usepackage{hyperref}
\usepackage{float}
\usepackage{booktabs}
\usepackage{amssymb}
\usepackage{multirow}
\usepackage{dsfont}
\usepackage{appendix}
\usepackage{soul}
\usepackage{amsmath}
\usepackage{pifont}
\usepackage{tikz}
\newcommand*{\circled}[1]{\lower.7ex\hbox{\tikz\draw (0pt, 0pt)%
    circle (.5em) node {\makebox[1em][c]{\small #1}};}}
\DeclareMathAlphabet{\mathcal}{OMS}{cmsy}{m}{n}
\DeclareSymbolFont{largesymbols}{OMX}{cmex}{m}{n}
\usepackage{graphicx}
\usepackage{tabularx}
\usepackage{epstopdf} % needed if you have eps figures in a pdflatex manuscript
\usepackage{mathptmx}      % use Times fonts if available on your TeX system
\biboptions{numbers,sort&compress}
\usepackage{chngcntr}
\usepackage{graphicx,psfrag}                    % 图形宏包

\renewcommand{\thefootnote}{\fnsymbol{footnote}}
\newtheorem{theorem}{Theorem}[section]

\newtheorem{lemma}{Lemma}[section]

\newtheorem{remark}{Remark}[section]

\newcommand{\ep}{\hfill\rule{0.15cm}{0.35cm}\vskip 0.3cm}
\newenvironment{proof}[1][Proof.]{ \begin{trivlist}
\item[\hskip \labelsep {\bfseries #1}]}{\ep\end{trivlist}}

\numberwithin{equation}{section}
\begin{document}

\begin{frontmatter}

\title{Finite element method for a
constant time delay subdiffusion equation with Riemann-Liouville fractional derivative}
\renewcommand{\thefootnote}{\fnsymbol{footnote}}
\author[xtu]{Weiping Bu\corref{cor}}
\author[xtu]{Chen Nie}
\author[xtu]{Weizhi Liao}
\address[xtu]{School of Mathematics and Computational Science$\And$Hunan Key Laboratory for Computation and Simulation in Science and Engineering,
Xiangtan University, Hunan 411105, China}
\cortext[cor]{Corresponding author.}
\begin{abstract}
This work considers to numerically solve a subdiffusion equation involving constant time delay $\tau$ and Riemann-Liouville fractional derivative. First, a fully discrete finite element scheme is developed for the considered problem under the symmetric graded time mesh, where the Caputo fractional derivative is approximated via the L1 formula, while the Riemann-Liouville integral is discretized using the fractional right rectangular rule. Under the assumption that the exact solution has low regularities at $t=0$ and $\tau$, the local truncation errors of both the L1 formula and the fractional right rectangular rule are analyzed. It is worth noting that, by setting the mesh parameter $r=1$, the symmetric graded time mesh will degenerate to a uniform mesh. Consequently, we proceed to discuss the stability and convergence of the proposed numerical scheme under two scenarios. For the uniform time mesh, by introducing a discrete sequence $\{P_k\}$, the unconditional stability and local time error estimate for the developed scheme is established. Conversely, on the symmetric graded time mesh, through the introduction of a discrete fractional Gronwall inequality, the stability and globally optimal time error estimate can be obtained. Finally, some numerical tests are presented to validate the theoretical results.
\end{abstract}
\end{frontmatter}

\section{Introduction}
\setlength{\parindent}{2em}

Since fractional calculus can successfully capture the memory and hereditary properties that
exist in a variety of materials and processes, it has attracted numerous scholars to
investigate its application, theory and numerical method \cite {Hilfer2000Applications,SUN2018collection,JIN2019Numerical,Li2019book,Kilbas2006book}. At the same time, it is worth pointing out that time delay often occurs in the real world, and is widely used in many fields of life sciences \cite {Schiesser2019book,Smith2010book,Polyanin2024book,Erneux2009book}. Recently, researchers have found that fractional delay differential equations can reflect both the historical memory effect and the response lag characteristics of the system, and thus have higher accuracy and adaptability in some fields such as biology, control and engineering \cite {book,Carvalho2017HIV/AIDS, RIHAN2021Fractional,NIRMALA201687,Rakhshan2020Fractional,LI20241Modeling}.

In this paper, we discuss to numerically solve the following constant time delay time fractional diffusion equation with Riemann-Liouville fractional derivative
    \begin{equation}\label{eq:main-equation}
    \begin{aligned}
         \partial_{t}u(x,t) &= \partial_{t}^{1-\alpha}\left(p\Delta u + au(x,t)\right) + bu(x,t-\tau) + f(x,t),\quad (x,t) \in \Omega_T,
    \end{aligned}
    \end{equation}
  with the initial and boundary conditions
  \begin{equation}\label{eq:initial equation}
      \begin{aligned}
          u(x,t) &= \varphi(x,t), \quad(x,t) \in \Omega \times [-\tau,0],
      \end{aligned}
  \end{equation}
  \begin{equation}\label{eq:boundary conditions}
      \begin{aligned}
\left.u(x,t)\right|_{\partial\Omega} &= 0,\quad t \in [-\tau,K\tau],
      \end{aligned}
  \end{equation}
where $0 < \alpha < 1$, $\Omega_T = (0,L) \times (0,K\tau]$, $\Omega = (0,L)$, $K \in \mathbb{N}^+$, $p > 0$, $a \leq 0$, $b \neq 0$ are some constants, $\tau > 0$ denotes the delay term, $f(x,t)$ is continuous on $\Omega_T$, $\varphi(x,t)$ is continuous on $\Omega \times [-\tau,0],$ and $\partial_t$ denotes the first derivative operator with respect to $t$. Here the Riemann-Liouville fractional derivative of order $1-\alpha$ is defined by
\begin{equation*}\label{eq:RL-fractional-derivative}
    \partial_t^{1-\alpha}u(x,t) = \partial_t {_0I_t^\alpha u(x,t)}, \quad 0 < \alpha < 1,
\end{equation*}
with the Riemann-Liouville fractional integral
\begin{equation*}
    _0I_t^\alpha u(x,t)=\int_0^t \omega_\alpha(t-s) u(x,s) \, ds
\end{equation*}
where the kernel function $\omega_\alpha(t) := \frac{t^{\alpha-1}}{\Gamma(\alpha)}$ and $\Gamma(\cdot)$ denotes the Gamma function.

In recent years, there have been a lot of theoretical and numerical studies on time fractional diffusion equations with time delay. In \cite{Zhu-2019-Existence,Ouyang-2011-Existence}, the existence and uniqueness of the solutions for nonlinear time fractional diffusion equations with time delay are discussed. And then, Ref. \cite{Karel-2022-existence} investigates the existence and uniqueness of the solution to a variable order time fractional delay diffusion equation. In \cite{Prakash-2020-Exact}, Prakash et al. propose the invariant subspace method to derive the exact solution of a nonlinear time fractional diffusion equation with time delay.
Hendy et al. \cite{Hendy-2023-Theoretical} investigate the long-time behavior of time fractional delay diffusion equations including the asymptotic stability and contractivity of the solution. Yao and Yang \cite{Yao-2023-Stability} obtain the asymptotical stability and long-time decay rates of the solutions to a fractional delay diffusion-wave equation. For a class of nonlinear fractional reaction-diffusion equations with delay. Shah and Irshad \cite{Shah-2025-Ulam-Hyers-Mittag-Leffler} use the fixed-point method to investigate the Ulam-Hyers-Mittag-Leffler stability of the fractional problem on a compact interval with respect to the Chebyshev and Bielecki norms.

In order to numerically solve the time fractional delay diffusion equations, Refs. \cite {Zhang-2017-Analysis,LI-2018-Convergence} propose the compact finite difference method.
Kumar et al. \cite {Kumar-2021-Numerical} develop a stable finite difference method for the time fractional convection-diffusion equation with singular perturbation and time delay.
Zhao et al. \cite {Zhao2018fast} devise a fast second-order implicit scheme for the time-space fractional delay diffusion equation.
By employing the finite difference/spectral Galerkin method, Hendy et al. \cite {Hendy-2023-Theoretical} obtain an effective numerical scheme, and prove that the numerical solution has the ability to preserve asymptotic contractivity and stability.
For solving the time fractional delay diffusion-wave equation, the finite difference/meshless method and Crank-Nicolson difference/cubic-trigonometric splines method are developed in \cite {Abbaszadeh-2019-Numerical,Chawla-2025-Numerical}.
In \cite {Peng2022ConvergenceAS,Peng2024Unconditionally}, the finite element method is devised for two time fractional delay diffusion equations, and the convergence and superconvergence are investigated.
In \cite {Farhood-2023-Solving,ZAKY2023L1}, the shifted Legendre-Laguerre operational matrices method and L1 type difference/Galerkin spectral method are employed for the variable-order time fractional delay diffusion equation.
In \cite {Pimenov-2017-non-linear,Abbaszadeh2021Galerkin}, the finite difference method and Galerkin meshless reproducing kernel particle method are used to solve the time distributed-order diffusion equation with delay.
Besides, we also note that a novel discrete Gronwall inequality based on the \(L2-1\sigma\) formula is developed and used in the numerical analysis of the time fractional delay diffusion equation \cite{Hendy2019Gronwall,Zaky2021Numerical}.

It is worth pointing out that a fractional ordinary differential equation with delay  is
investigated in \cite{Morgado2013analysis} investigate, and the derived exact solution implies that it exhibits multi-singularity at $k\tau$ with time delay $\tau$.
Subsequently, the multi-singularities of time fractional delay diffusion equation and delay fractional ordinary differential equation are rigorously demonstrated in \cite{Tan,CenVong-2023-Tracking,Bu-2024-Finite,OU2024Variable}, which show that
\begin{equation}\label{eqs250718A}
|\partial_{t}u_{k\tau}|\leq  C\left(1+(t-(k-1)\tau)^{k\alpha-1}\right),
\end{equation}
where \(u_{k\tau}(x,t)\) represents \(u(x,t)\) for \(t\in((k - 1)\tau,k\tau]\).
Furthermore, in order to overcome the temporal multi-singularity, Refs. \cite{Tan,Bu-2024-Finite,OU2024Variable} used the symmetrical graded mesh to develop numerical scheme to achieve global optimal time convergence rate for time fractional delay diffusion equation, and corrected $L$-type method is proposed for delay fractional ordinary differential equation \cite{CenVong-2023-Tracking}.
When the globally time-optimal convergent numerical schemes have been established, this naturally motivates researchers to focus on developing numerical methods with pointwise-in-time error estimates.
In \cite{Liang-2023-Pointwise}, the discrete Laplace transform method is used to obtain pointwise error estimate of L1 method of delay fractional ordinary differential equation. Meanwhile, a new discrete Gronwall inequality is established in \cite{Bu-2025-Local} to achieve the local temporal convergence rate of L1/finite element scheme for time fractional delay diffusion equation.

Recently, when investigating the temporal regularity of the solution to (\ref{eq:main-equation})--(\ref{eq:boundary conditions}), we find that \cite{Bu-2025-Finite}
\begin{equation}\label{eqs250719A}
|\partial_{t}u_{k\tau}|\leq C\left(1+t^{\alpha-1}\right),\quad k=1,2,\cdots,K,
\end{equation}
and  the second time derivative satisfies
\begin{equation}\label{eqs250719B}
|\partial_{t}^2u_{\tau}|\leq C\left(1+t^{\alpha-2}\right) \quad
\text{and} \quad
|\partial_{t}^2u_{k\tau}|\leq C\left(1+(t-\tau)^{\alpha-1}\right),\quad k=2,3,\cdots K,
\end{equation}
where
\(u_{\tau}(x,t):=u_{1\tau}(x,t)\), and $\partial^2_t$ denotes the second derivative operator
with respect to $t$.
The above results mean that  the solution has distinct and superior regularity in time compared with that mentioned in Refs. \cite{Tan,Bu-2024-Finite,CenVong-2023-Tracking,OU2024Variable} satisfying (\ref{eqs250718A}).
Therefore, it motivates us to develop effective numerical method to solve  (\ref{eq:main-equation})--(\ref{eq:boundary conditions}).
The main contribution of this work includes:

$\bullet$
A fully discrete finite element scheme is established for (\ref{eq:main-equation})--(\ref{eq:boundary conditions}) by employing
L1 formula for Caputo fractional derivative and the fractional right rectangle
rule for Riemann-Liouville fractional integral on the symmetric graded time mesh,
which degenerates to a uniform mesh when the mesh parameter $r=1$.

$\bullet$
Under the new regularity assumptions (\ref{eqs250719A})--(\ref{eqs250719B}) and
the symmetric graded time mesh, the local truncation errors of L1 formula and fractional right rectangle rule are investigated.

$\bullet$
For the case of uniform time
mesh, by introducing a discrete sequence $\{P_k\}$, we obtain the stability and local time error estimate for the developed numerical scheme without the increasing Mittag-Leffler function. Notably, local time error estimate implies that the maximum temporal error order is $\alpha$ for the interval $(0,\tau]$ and 1 for $(\tau,K\tau].$
It is clear that due to the improved regularity,
the finite element scheme potentially yields a superior time
accuracy near $t=(i\tau)^+,i\geq1$ than that attained by \cite{Bu-2025-Local}.

$\bullet$
For the symmetric graded time mesh,
by applying a classical discrete fractional Gronwall inequality, the stability and global convergence order $\min\{1, r\alpha\}$ are obtained.
The convergence result shows that the global convergence order is $\alpha$ for $r=1$, while
maximum time error order can reach 1 on the
interval $(\tau, K\tau]$ for the aforementioned case. Therefore, the convergence results for the aforementioned uniform mesh do not constitute a special case of the present convergence findings.
\\
Besides, it is noteworthy that the convergence behavior observed under both uniform and symmetric graded time meshes demonstrates that the accuracy of numerical solution can be flexibly adjusted between local and global time convergence orders by tuning the mesh parameter $r$.

The structure of the remainder of this paper is organized as follows.
In Section 2,
the fully discrete finite element scheme for (\ref{eq:main-equation})--(\ref{eq:boundary conditions}) is established,
and the local truncation errors of L1 formula and fractional right rectangular rule are discussed.
Then, based on uniform time mesh,
the stability and local time error estimate of the developed numerical scheme are obtained in section 3.
In Section 4, the stability and global time
error estimate of the proposed numerical scheme are investigated under the symmetrical graded time mesh.
In Section 5, some numerical tests are provided to verify that the correctness of
numerical theory.

{\bf Notation}. In the following paper, $C$ represents a positive number that  can  different in different situations, and  $a \lesssim b$ means $a\leq C b.$

\section{ Numerical scheme of \eqref{eq:main-equation}--\eqref{eq:boundary conditions}}
In order to develop an efficient numerical scheme for (\ref{eq:main-equation})--(\ref{eq:boundary conditions}), similar to \cite{Bu-2025-Finite}, we firstly transform (\ref{eq:main-equation}) into
\begin{equation}\label{eq:After Integration}
\begin{aligned}
{}_0^C\mathrm{D}_t^{\alpha} u(x,t) &= p\Delta u + au(x,t) + {}_0I_t^{1 - \alpha}\left[bu(x,t - \tau)\right] + G(x,t), \quad (x,t) \in \Omega_T, \\
\end{aligned}
\end{equation}
where
$G(x,t)={}_0I_t^{1 - \alpha}f(x,t),$ and the Caputo fractional derivative is defined by
\[{}_0^C\mathrm{D}_t^{\alpha}u(x,t):=\int_{0}^{t} \omega_{1-\alpha} (t-s)\partial_s u(x,s)ds.\]
To divide the time interval $[-\tau, K\tau]$, the following symmetric graded mesh is used \cite{Bu-2024-Finite}
\begin{equation}\label{eq:250906a}
t_{-2N}=-\tau,\quad t_{n}=
\begin{cases}
\frac{\tau}{2}\left(\frac{n - 2(i - 1)N}{N}\right)^{r}+(i - 1)\tau, & 2(i - 1)N + 1\leq n < (2i - 1)N,\\
-\frac{\tau}{2}\left(\frac{2iN - n}{N}\right)^{r}+i\tau, & (2i - 1)N\leq n\leq 2iN,
\end{cases}
\end{equation}
where $i = 0, 1, 2, \cdots, K$, $N\geq2$ is a integer, the mesh parameter $r\geq 1$.
It is noteworthy that when $r=1$, the symmetric graded mesh (\ref{eq:250906a}) will
degenerate into uniform time mesh.
Let $\rho_n=t_n - t_{n - 1}$,  $u^n = u(t_n)$ and $\nabla_tu^n=u^n-u^{n-1}$. According to \cite{Bu-2024-Finite}, the aforementioned  symmetric graded mesh satisfies
\begin{equation}\label{eq:graded mesh satisfies}
    \rho_{1}=t_{1}\simeq N^{-r},\quad \rho_{n}\lesssim (\rho_{1})^{1/r}(t_{n}-(i - 1)\tau)^{1 - \frac{1}{r}}, \ 2(i-1)N+1\leq n \leq2iN, \ i=1,2,\cdots,K.
\end{equation}
By \cite{Li2019book}, we introduce the L1 approximation of Caputo fractional derivative
\begin{equation}\label{eq:discretization of Caputo}
    D^\alpha u^n=\sum_{k=1}^{n}a_{n-k}^{(n)}\nabla_t u^k,
\end{equation}
and the fractional right rectangle formula of Riemann-Liouville  fractional integral
\begin{equation}\label{eq:J825}
    J^{1-\alpha}u^{n-2N}=\sum_{k=1}^n\rho_k a_{n-k}^{(n)}u^{k-2N},
\end{equation}
where $1\leq n\leq 2KN$ and
\begin{equation}\label{eqs250829A}
    a_{n-k}^{(n)} = \frac{\omega_{2-\alpha}(t_{n}-t_{k-1}) - \omega_{2-\alpha}(t_{n}-t_{k})}{\rho_k}.
\end{equation}

Since the discretization of the Riemann-Liouville fractional integral adopts the ‌fractional right rectangle formula‌ instead of the fractional trapezoidal formula employed in works such as
\cite{Bu-2024-Finite}, and the ‌regularity of the solution‌ given herein differs from the previous cases, we have to re-examine the ‌local truncation error‌ of the fractional right-rectangle formula and the L1 formula under the regularity conditions \eqref{eqs250719A} and \eqref{eqs250719B}.

\begin{lemma}\label{lem: Truncation error of J on the hierarchical grid}
Assume that  $u$  satisfies the regularity assumption (\ref{eqs250719A}) and $\partial_t\varphi \in L^{\infty}[-\tau,0]$. Then for the symmetric grade mesh, it holds
\begin{equation}\label{eqs250824B}
\begin{aligned}
    \left|(J^{1-\alpha}-_{0}{I}_t^{1-\alpha})u^{n-2N}\right| &\lesssim
        \rho_1^{\frac{1}{r}}, & 1 \leq n \leq 2KN.
\end{aligned}
\end{equation}
\end{lemma}
\begin{proof}
Here we focus mainly on the case \(n \geq 2N+1\).
Define
\(m(t) = u(t - \tau) - u^{k - 2N}\), \(t\in(t_{k-1},t_k),\ k \geq 1. \)
If $k\neq 2N+1$, then
 the  mean value theorem gives
\begin{equation}\label{eq:m(s)1}
    m(t) \lesssim \rho_k \partial_t u(\xi_k - \tau),\quad t_{k-1} < \xi_k <t_k;
\end{equation}
and if $k=2N+1$, then
\begin{equation}\label{eq:m(s)2}
\begin{aligned}
    |m(t)| &= \left|\int_{t}^{t_{2N+1}} \partial_sm(s) \, ds \right|\\
    &=\left|\int_{t}^{t_{2N+1}}(s-\tau)^{\alpha-1}\left[(s-\tau)^{1-\alpha}\partial_su(s-\tau)\right]ds\right|\\
    &\lesssim(t-\tau)^{\alpha-1}(t_{2N+1}-t),
\end{aligned}
\end{equation}
where $\partial_t u(\xi_k - \tau)$ denotes $\partial_t u(t - \tau)|_{t=\xi_k}$ and the regularity (\ref{eqs250719A}) is used.
According to (\ref{eq:m(s)1}) and (\ref{eq:m(s)2}),
we split \(\left\lvert \left(J^{1-\alpha}{-}_{0}I^{1-\alpha}_{t}\right)u^{n - 2N} \right\rvert\) into the following three parts
\begin{equation}\label{eqs250724M}
\begin{aligned}
&\left\lvert  \left(J^{1-\alpha}{-}_{0}I^{1-\alpha}_{t}\right)u^{n - 2N} \right\rvert \\
\lesssim &\left\lvert \sum_{k=1}^{2N} \int_{t_{k-1}}^{t_k} (t_n - t)^{-\alpha} m(t) \, dt \right\rvert
+ \left\lvert \int_{t_{2N}}^{t_{2N+1}} (t_n - t)^{-\alpha} m(t) \, dt \right\rvert + \left\lvert \sum_{k=2N+2}^n \int_{t_{k-1}}^{t_k} (t_n - t)^{-\alpha} m(t) \, dt \right\rvert \\
\lesssim &\left\lvert \sum_{k=1}^{2N} \int_{t_{k-1}}^{t_k} (t_n - t)^{-\alpha} \rho_k \partial_t u(\xi_k - \tau) \, dt \right\rvert
+\int_{t_{2N}}^{t_{2N+1}} (t_n - t)^{-\alpha}(t-\tau)^{\alpha-1}(t_{2N+1}-t) \, dt \\
&+ \left\lvert \sum_{k=2N+2}^n \int_{t_{k-1}}^{t_k} (t_n - t)^{-\alpha}  \rho_k \partial_t u(\xi_k - \tau) \, dt \right\rvert\\
:=& \mathrm{I} +  \mathrm{II} +  \mathrm{III}.
\end{aligned}
\end{equation}

First it follows from $\partial_t\varphi \in L^{\infty}[-\tau,0]$ that
\begin{equation}\label{eqs250826A}
\begin{aligned}
    \mathrm{I}
&= \left\lvert \sum_{k=1}^{2N} \int_{t_{k-1}}^{t_k} (t_n - t)^{-\alpha} \rho_k \partial_t \varphi(\xi_k - \tau) \, dt \right\rvert \\
&\lesssim  \rho_1^{\frac{1}{r}} \sum_{k=1}^{2N} \int_{t_{k-1}}^{t_k} (t_n - t)^{-\alpha}t_k^{1-\frac{1}{r}} \, dt \\
&\lesssim  \rho_1^{\frac{1}{r}}t_{2N}^{1-\frac{1}{r}}  \int_{{0}}^{t_{2N}} (t_n - t)^{-\alpha} \, dt\\
&\lesssim \rho_1^{\frac{1}{r}}.
\end{aligned}
\end{equation}
For $t\in(t_{2N},t_{2N+1})$, applying the fact
\[\frac{t_n-\tau}{t_{2N+1}-\tau}\lesssim\frac{t_{n}-t}{t_{2N+1}-t},
\quad \text{i.e.},\quad
    t_n - \tau \lesssim \frac{t_n - t}{t_{2N+1} - t} (t_{2N+1} - \tau),\]
then
\begin{equation}\label{eqs250826B}
\begin{aligned}
\mathrm{II}
&= \int_{t_{2N}}^{t_{2N+1}} (t_n - t)^{-\alpha} (t_{2N+1} - t) (t - \tau)^{\alpha-1}  \, dt \cdot \frac{(t_n - \tau)^\alpha}{\rho_1^{\alpha+1}} \frac{\rho_1^{\alpha+1}}{(t_n - \tau)^\alpha} \\
&\lesssim \int_{t_{2N}}^{t_{2N+1}} (t_{2N+1} - t)^{1-\alpha} (t - \tau)^{\alpha-1} \, dt \cdot \frac{1}{t_{2N+1} - \tau} \frac{(t_{2N+1} - \tau)^{\alpha+1}}{\rho_1^{\alpha+1}}  \frac{\rho_1^{\alpha+1}}{(t_n - \tau)^\alpha} \\
&= \int_{t_{2N}}^{t_{2N+1}} \left( 1- \frac{t - \tau}{t_{2N+1} - \tau}\right)^{1-\alpha} \left( \frac{t - \tau}{t_{2N+1} - \tau} \right)^{\alpha-1} \, d \left(\frac{t-\tau}{t_{2N+1} - \tau} \right)\cdot\frac{\rho_1^{\alpha+1}}{(t_n - \tau)^\alpha}\\
&\lesssim \frac{\rho_1^{\alpha+1}}{(t_n - \tau)^\alpha}.
\end{aligned}
\end{equation}
Since
\((t_k - \tau)^{1-\alpha} \sup\limits_{t \in (t_{k-1}, t_{k})}| \partial_t u(t - \tau) |\lesssim 1\) for \(k\geq2N+2\),
it is obvious that
\begin{equation}\label{eqs250826C}
    \begin{aligned}
        \mathrm{III}
&\lesssim  \sum_{k=2N+2}^{n} \int_{t_{k-1}}^{t_k} (t_n - t)^{-\alpha} \rho_k (t_k - \tau)^{\alpha-1} \left[(t_k - \tau)^{1-\alpha} |\partial_t u(\xi_k - \tau)|\right] \, dt \\
&\lesssim  \sum_{k=2N+2}^{n} \int_{t_{k-1}}^{t_k} (t_n - t)^{-\alpha} (t_k - \tau)^{\alpha-1} \, dt\cdot \max\limits_{2N+2 \leq k\leq n}\rho_k \\
&= \int_{t_{2N+1}}^{t_n} \left( 1-\frac{t - \tau}{t_n - \tau}
\right)^{-\alpha} \left( \frac{t- \tau}{t_n - \tau} \right)^{\alpha-1} \, d \left(\frac{t-\tau}{t_n - \tau}\right)\cdot \max\limits_{2N+2 \leq k\leq n}\rho_k\\
&\lesssim \rho_1^{\frac{1}{r}}.
    \end{aligned}
\end{equation}

Based on \eqref{eqs250826A}--\eqref{eqs250826C}, we can conclude that for \(n \geq 2N + 1\), $\left| (J^{1-\alpha}{-}_{0}I^{1-\alpha}_{t}) u^{n-2N} \right| $ can be bounded by
\begin{equation}\label{eqs250826D}
    \begin{aligned}
        \left| (J^{1-\alpha}{-}_{0}I^{1-\alpha}_{t}) u^{n-2N} \right|
        \lesssim
        \rho_1^{\frac{1}{r}}.
    \end{aligned}
\end{equation}
When  \(1 \leq n \leq 2N\), it is clear that
\begin{equation}\label{eqs250826E}
    \begin{aligned}
        &\left| (J^{1-\alpha}{-}_{0}I^{1-\alpha}_{t}) u^{n-2N} \right| \\
       \lesssim  &\rho_1^{\frac{1}{r}}\left|\int _0^{t_n} (t_n-t)^{-\alpha}  \partial_t\varphi (\xi_k-\tau)dt\right|\\
        \lesssim & \rho_1^{\frac{1}{r}}.
    \end{aligned}
\end{equation}
Thus the combination of (\ref{eqs250826D}) and (\ref{eqs250826E}) indicates that  (\ref{eqs250824B}) holds.
\end{proof}

Define  $ \psi_k$ as
\begin{equation}\label{eq:psi-definitions}
    \begin{aligned}
        \psi_{k} =
        \begin{cases}
            \sup\limits_{t\in(t_{k-1},t_{k})} t^{1 - \alpha} \left| \frac{\nabla_{t}u^{k}}{\rho_{k}} - \partial_{t}u \right|, \quad k = 1,\\
              t_{k}^{1 + \alpha} \sup\limits_{t\in(t_{k - 1},t_{k})}| \partial_{t}^{2}u|, \quad 2 \leq k \leq  2N,\\
              (t_k - \tau)^{1+\alpha} \sup\limits_{t \in (t_{k-1}, t_k)} \vert \partial_t^2 u \vert, \ 2N + 2 \leq k \leq n.
        \end{cases}
    \end{aligned}
\end{equation}
Then the following property holds.
\begin{lemma}\label{lem: grade psi approximate value}
Suppose that \(u\) is the solution of \eqref{eq:main-equation}
and satisfies the regularity \eqref{eqs250719A}--\eqref{eqs250719B}.
Under the symmetric grade mesh \eqref{eq:250906a}, we have
 \begin{equation}
    \begin{aligned}
        \psi_k\lesssim
        \begin{cases}
            1,           &\quad  k =1,\\
            t_k^{2\alpha-1},  &\quad 2 \leq  k \leq 2N,\\
            (t_k - \tau)^{2\alpha}, &\quad 2N + 2 \leq k \leq n.
        \end{cases}
    \end{aligned}
\end{equation}
\end{lemma}
\begin{proof}
Since
\begin{equation}\label{es250806A}
    \begin{aligned}
        \left|\frac{\nabla u^1}{\rho_1}\right|&= \left|\frac{\int_{0}^{t_1}\partial_tu dt}{\rho_1}\right|\\
%&\lesssim\frac{\int_0^{t_1}\frac{1+t^{\alpha-1}}{t^{\alpha-1}}t^{\alpha-1}dt}{\rho_1}\\
%&\lesssim(1+t_1^{1-\alpha})\frac{\int_{0}^{t_1} t^{\alpha - 1} dt}{\rho_1}\\
  &\lesssim
  \frac{\int_{0}^{t_1} t^{\alpha - 1} dt}{\rho_1} =  t_1^{ \alpha-1},
    \end{aligned}
\end{equation}
it follows from (\ref{eq:psi-definitions}) and \eqref{es250806A} that
\begin{align*}
     \psi_{1}&
     \lesssim \sup_{t\in(t_{0},t_{1})} t^{1-\alpha}(|t_1^{\alpha-1}|+|\partial_tu(t)|)\\
     &\lesssim \sup_{t\in(t_{0},t_{1})}t^{1 - \alpha}(t_1^{\alpha-1}+ t^{\alpha - 1})\lesssim 1.
\end{align*}
Besides, applying the regularity assumption (\ref{eqs250719B}) and the fact $t_k\lesssim t_{k-1}$ when $k\geq2$ and $t_k-\tau\lesssim t_{k-1}-\tau$ when $k\geq2N+2$, it is easy  to  obtain that
\[
\begin{aligned}
\psi_k &= t_k^{1+\alpha} \sup_{t \in (t_{k-1}, t_k )} \left| \partial_t^2 u \right|\\
&\lesssim t_k^{2\alpha-1}
\end{aligned}
\]
for \( k = 2, 3, \cdots, 2N \), and
\[
\begin{aligned}
\psi_k &= (t_k - \tau)^{1+\alpha} \sup_{t \in (t_{k-1}, t_k)} \vert \partial_t^2 u \vert \\
&\lesssim (t_k - \tau)^{1+\alpha}(t_k - \tau)^{\alpha - 1}\\
&\lesssim (t_k - \tau)^{2\alpha}
\end{aligned}
\]
for \( k = 2N + 2, 2N + 3, \cdots, n \).
\end{proof}

By employing the defined ‌$\psi_k$‌, now we derive the ‌local truncation error of L1 formula.
\begin{lemma}\label{lem: Truncation error of D on the hierarchical grid}
Suppose that \(u\) is the solution of \eqref{eq:After Integration}
and satisfies the regularity \eqref{eqs250719A}--\eqref{eqs250719B}.
Then under the symmetric grade mesh \eqref{eq:250906a}, one has
\begin{equation}\label{eq:truncation error of D on the symmetrical grid}
\left|({D}^{\alpha}-{_0^C}D_{t}^{ \alpha})u^{n}\right| \lesssim
\begin{cases}
\begin{aligned}
&\left(\frac{\rho_{1}}{t_{n}}\right)^{\min\{\frac{2 - \alpha}{r},\alpha + 1\}},\ 1 \leq n \leq 2N ,\\
& \rho_1^{1+\alpha}(t_n-\tau)^{-\alpha}+\rho_1^{\frac{2-\alpha}{r}}(t_n-\tau)^{\frac{\alpha-2}{r}+1},  \ 2N + 1 \leq n \leq 2KN.\\
\end{aligned}
\end{cases}
\end{equation}
\end{lemma}

\begin{proof}
%当$n\in[4N+1,KN]$时，利用分部积分公式，
When \(2N + 1 \leq n \leq 2KN\), splitting \( \left| \left( D^{\alpha} -{_0^C}D_t^{\alpha}\right)u^n \right|\) into following three parts
\begin{equation}\label{eqs250901C}
\begin{aligned}
    \left| \left( D^{\alpha} -{_0^C}D_t^{\alpha}\right)u^n \right| &= \left| \frac{1}{\Gamma(1-\alpha)} \sum_{k=1}^n \int_{t_{k-1}}^{t_k} (t_n - t)^{-\alpha} \left( \partial_t u(t)-\frac{\nabla_t u^k}{\rho _k} \right) dt \right| \\
    &= \left| \frac{\alpha}{\Gamma(1-\alpha)} \sum_{k=1}^n \int_{t_{k-1}}^{t_k} (t_n - t)^{-\alpha-1} (u - u_L) dt \right| \\
    &\lesssim \left|\sum_{k=1}^2 \int_{t_{2(k-1)N}}^{t_{2(k-1)N+1}}(t_n-t) ^{-\alpha-1}(u - u_L)dt \right|+ \left|\sum_{k=2}^{2N }\int_{t_{k-1}}^{t_k} (t_n-t) ^{-\alpha-1}(u - u_L)dt\right|\\
    &\quad+\left| \sum_{k=2N+2}^n \int_{t_{k-1}}^{t_k} (t_n-t) ^{-\alpha-1}(u - u_L)dt\right|,
\end{aligned}
\end{equation}
where integration by parts is used, and $u_L$ denotes the linear Lagrange interpolation function of $u$ on $t\in[t_{k-1}, t_k],k\geq1$.
It is easy to check that \eqref{eqs250901C} can be rewritten as
\begin{equation}\label{eqs250828D}
\begin{aligned}
\left| \left( D^{\alpha} -{_0^C}D_t^{\alpha}\right)u^n \right|&\lesssim \left|\sum_{k=1}^2 \int_{t_{2(k-1)N}}^{t_{2(k-1)N+1}}(t_n-t) ^{-\alpha-1}\int_{t_{2(k-1)N+1}}^{t}\left( \partial_su(s)-\frac{\nabla_t u^{{2(k-1)N+1}}}{\rho _{2(k-1)N+1}}\right)dsdt \right|\\
    &\quad + \left|\sum_{k=2}^{2N }\int_{t_{k-1}}^{t_k} (t_n-t) ^{-\alpha-1}\int_{t_k}^{t}\left( \partial_su(s)-\frac{\nabla_t u^k}{\rho _k}\right)dsdt\right|\\
    &\quad+\left| \sum_{k=2N+2}^n \int_{t_{k-1}}^{t_k} (t_n-t) ^{-\alpha-1}\int_{t_k}^{t}\left( \partial_su(s)-\frac{\nabla_t u^k}{\rho _k}\right)dsdt\right|\\
    &:= \mathrm{IV}+\mathrm{V}+\mathrm{VI}.
\end{aligned}
\end{equation}

Note that
\begin{equation}\label{eq250906b}
\begin{aligned}
\mathrm{IV}\leq&\left| \int_{t_{0}}^{t_{1}}(t_n-t) ^{-\alpha-1}\int_{t_{1}}^{t}\left( \partial_su(s)-\frac{\nabla_t u^{{1}}}{\rho _{1}}\right)dsdt \right|\\
 &+ \left| \int_{t_{2N}}^{t_{2N+1}}(t_n-t) ^{-\alpha-1}\int_{t_{2N+1}}^{t}\left( \partial_su(s)-\frac{\nabla_t u^{{2N+1}}}{\rho _{2N+1}}\right)dsdt \right|.
\end{aligned}
\end{equation}
From (\ref{eq:psi-definitions}), the first term on the right hand side of (\ref{eq250906b}) yields
\begin{equation}\label{eqs250903B}
\begin{aligned}
    &\left| \int_{t_{0}}^{t_{1}}(t_n-t) ^{-\alpha-1}\int_{t_{1}}^{t}\left( \partial_su(s)-\frac{\nabla_t u^{{1}}}{\rho _{1}}\right)dsdt \right|\\
    \lesssim& \int_{t_{0}}^{t_{1}}(t_{n}-t)^{-\alpha - 1}t^{\alpha - 1}(t_{1}-t) \psi_{1}dt.
\end{aligned}
\end{equation}
By (\ref{eqs250719B}) and Taylor's formula with the integral form of the remainder, one has
\[
\begin{aligned}
    \left| \partial_su(s)-\frac{\nabla_t u^{{2N+1}}}{\rho _{2N+1}}\right|
    &\leq\frac{1}{\rho_1}\left(\int_{t_{2N}}^{s}|\partial^2_yu(y)|(y-t_{2N})dy
    +\int_{s}^{t_{2N+1}}|\partial^2_yu(y)|(t_{2N+1}-y)dy\right)\\
    &\leq\left(\int_{t_{2N}}^{s}(y-\tau)^{\alpha-1}dy
    +\int_{s}^{t_{2N+1}}(y-\tau)^{\alpha-1}dy\right)\\
    &\lesssim\rho_1^\alpha,\quad s\in (t_{2N},t_{2N+1}).
\end{aligned}
\]
Therefore, the second term on the right hand side of (\ref{eq250906b}) can be bounded as
\begin{equation}\label{eqs250903C}
\begin{aligned}
    &\left| \int_{t_{2N}}^{t_{2N+1}}(t_n-t) ^{-\alpha-1}\int_{t_{2N+1}}^{t}\left( \partial_su(s)-\frac{\nabla_t u^{{2N+1}}}{\rho _{2N+1}}\right)dsdt \right|\\
    \lesssim& \rho_1^\alpha \int_{t_{2N}}^{t_{2N+1}}(t_n-t) ^{-\alpha}(t_n-t)^{-1} (t_{2N+1}-t)dt\\
    \lesssim&\rho_1^\alpha \int_{t_{2N}}^{t_{2N+1}} (t_{n}-t)^{-\alpha}dt\\
    \lesssim& \rho _1^{1+\alpha}(t_n-\tau)^{-\alpha},
\end{aligned}
\end{equation}
where $t_n-t_{2N}\lesssim t_n-t_{2N+1}$ for $n\geq2N+2$ is used.
%and if $$, then
%\begin{equation}\label{eqs250903D}
%    \begin{aligned}
%        &\left| \int_{t_{2N}}^{t_{2N+1}}(t_{2N+1}-t) ^{-\alpha-1}\int_{t_{2N+1}}^{t}\left( \partial_su(s)-\frac{\nabla_t u^{{2N+1}}}{\rho _{2N+1}}\right)dsdt \right|\\
%        \lesssim& \rho_1^\alpha \int_{t_{2N}}^{t_{2N+1}}(t_{2N+1}-t) ^{-\alpha-1} (t_{2N+1}-t)dt\\
%        %\lesssim&\rho_1^\alpha \int_{t_{2N}}^{t_{2N+1}} (t_{2N+1}-t)^{-\alpha}dt\\
%        \lesssim& \rho _1^{1+\alpha}(t_{2N+1}-\tau)^{-\alpha}.
%    \end{aligned}
%\end{equation}
%thus \(IV\lesssim \left(\frac{\rho_1}{t_n}\right)^{\alpha+1}+ \rho_1\).
The above discussion implies that
\begin{equation}\label{eqs250908A}
\begin{aligned}
\mathrm{IV}\leq \int_{t_{0}}^{t_{1}}(t_{n}-t)^{-\alpha - 1}t^{\alpha - 1}(t_{1}-t) \psi_{1}dt+\rho _1^{1+\alpha}(t_n-\tau)^{-\alpha}.
\end{aligned}
\end{equation}
When $k\neq1,2N+1$,
according to Lagrange's mean value theorem, it is clear that
\begin{equation}\label{eqs250901B}
    \begin{aligned}
    \left|\partial_su(s)-\frac{\nabla_t u^k}{\rho _k}\right|\leq \rho_k\partial_t^2u(\xi_k),
    \end{aligned}
\end{equation}
where $s, \ \xi_k \in(t_{k-1},t_k)$. By employing (\ref{eq:psi-definitions}) and \eqref{eqs250901B},
it indicates that $\mathrm{V}$ and $\mathrm{VI}$ can be bounded by
\begin{equation}\label{eqs250901D}
\begin{aligned}
\mathrm{V} &\lesssim \sum_{k = 2}^{2N}\int_{t_{k - 1}}^{t_{k}}(t_{n}-t)^{-\alpha - 1}\rho _k^2 t_k^{-1-\alpha}t_k^{1+\alpha}|\partial_t^2u(\xi_k)|\,dt\\
&\lesssim \int_{t_{1}}^{t_{2N}}(t_{n}-t)^{-\alpha - 1}t^{-\frac{\alpha}{r}-1}\,dt \times \max_{2 \leq k\leq 2N}\left\{\rho_{k}^{2}t_{k}^{\frac{\alpha}{r}-\alpha}\psi_{k}\right\},
\end{aligned}
\end{equation}
and
\begin{equation}\label{eqs250903A}
\begin{aligned}
\mathrm{VI} &\lesssim \left| \sum_{k=2N+2}^{n} \int_{t_{k-1}}^{t_k} (t_n - t)^{-\alpha - 1} \int_{t_k}^{t} \left( \partial_s u(s) - \frac{\nabla_t u^k}{\rho_k} \right) ds dt \right| \\
&\lesssim \sum_{k=2N+2}^{n-1} \int_{t_{k-1}}^{t_k} (t_n - t)^{-\alpha - 1} \rho_k^2 |\partial_t^2 u(\xi_k)| dt + \rho_n^2 \int_{t_{n-1}}^{t_n} (t_n - t)^{-\alpha - 1} \left( \frac{t_n - t}{\rho_n} \right) |\partial_t^2 u(\xi_n)| dt \\
&\lesssim \int_{t_{2N+1}}^{t_n} (t_n - t)^{-\alpha - 1} (t - \tau)^{-\frac{\alpha}{r} - 1} \min\left\{ 1, \frac{t_n - t}{\rho_n} \right\} dt \times \max_{2N+1 \leq k \leq n} \left\{ \rho_k^2 (t_k - \tau)^{\frac{\alpha}{r} - \alpha} \psi_k \right\}.
\end{aligned}
\end{equation}

From \cite[The proof of Lemma 7]{Tan}, the following inequalities hold
\begin{equation}\label{eqs250827A}
\begin{aligned}
    \int_{t_0}^{t_{1}}(t_{n}-t)^{-\alpha - 1}t^{\alpha - 1}(t_{1}-t) dt
\lesssim \left(\frac{\rho_1}{t_n}\right)^{\alpha+1},\\
\end{aligned}
\end{equation}
\begin{equation}\label{eqs250827B}
\begin{aligned}
\int_{t_{1}}^{t_{2N}}(t_{n}-t)^{-\alpha - 1}t^{-\frac{\alpha}{r}-1}\,dt
\lesssim \rho_1^{-\frac{\alpha}{r}}t_n^{-\alpha-1}, \\
\end{aligned}
\end{equation}
\begin{equation}\label{eqs250827C}
\begin{aligned}
\int_{t_{2N+1}}^{t_n} (t_n - t)^{-\alpha - 1}(t - \tau)^{-\frac{\alpha}{r} - 1}\min\left\{1,\frac{t_n - t}{\rho_n}\right\}dt \lesssim \rho_1^{-\frac{\alpha}{r}} (t_n - \tau)^{-\alpha - 1}.
\end{aligned}
\end{equation}
By replacing (\ref{eqs250908A}), (\ref{eqs250901D}) and (\ref{eqs250903A}) with these inequalities, and using Lemma \ref{lem: grade psi approximate value}, we can easily obtain the following estimates
\begin{equation}\label{eq250906c}
\begin{aligned}
\mathrm{IV}\lesssim \left(\frac{\rho_1}{t_n}\right)^{\alpha+1}
+\rho_1^{1+\alpha}(t_n-\tau)^{-\alpha},
\end{aligned}
\end{equation}
\begin{equation}\label{eq250906d}
\begin{aligned}
\mathrm{V}\lesssim& \rho_1^{-\frac{\alpha}{r}}t_n^{-\alpha-1}\rho_{1}^{\frac{2}{r}} \max_{1\leq k\leq2N}t_{k}^{\alpha+\frac{\alpha - 2}{r}+1}\\
\lesssim& \left(\frac{\rho_{1}}{t_{n}}\right)^{\min\left\{\frac{2 - \alpha}{r},\alpha + 1\right\}}
\end{aligned}
\end{equation}
and
\begin{equation}\label{eq250906e}
\begin{aligned}
\mathrm{VI}\lesssim&  \max_{2N+1 \leq k \leq n} \left\{ \rho_k^2 (t_k - \tau)^{\frac{\alpha}{r} + \alpha}  \right\} \rho_1^{-\frac{\alpha}{r}} (t_n - \tau)^{-\alpha - 1}\\
\lesssim& \rho_1^{\frac{2-\alpha}{r}}(t_n-\tau)^{\frac{\alpha-2}{r}+1}.
\end{aligned}
\end{equation}
Now, in virtue of (\ref{eq250906c})--(\ref{eq250906e}) and (\ref{eqs250828D}), we can conclude that for \(2N + 1 \leq n \leq 2KN\), the local truncation error of L1 formula satisfies
\begin{equation}\label{eqs250828A}
\begin{aligned}
\left| \left( D^{\alpha} - {_0^C}D_{t}^{ \alpha} \right) u^{n} \right|
\lesssim \rho_1^{1+\alpha}(t_n-\tau)^{-\alpha}+\rho_1^{\frac{2-\alpha}{r}}(t_n-\tau)^{\frac{\alpha-2}{r}+1}.
\end{aligned}
\end{equation}
Besides, from the above discussion, it is easy to find that
\begin{equation}\label{eqs250828B}
    \left| \left( D^{\alpha} - {_0^C}D_{t}^{ \alpha} \right) u^{n} \right| \lesssim \left( \frac{\rho_{1}}{t_{n}} \right)^{\min\left\{ \frac{2 - \alpha}{r}, \alpha + 1 \right\}} \ \text{for}  \ 1\leq n \leq 2N.
\end{equation}
Thus, the combination of \eqref{eqs250828A} and \eqref{eqs250828B} can obtain
 (\ref{eq:truncation error of D on the symmetrical grid}) immediately.
%\begin{comment}
%    \begin{equation*}
%    \left|({D}^{\alpha}-{_0^C}D_{t}^{1 - \alpha})u^{n}\right| \lesssim
%    \begin{cases}
%        \begin{aligned}
%            &\left(\frac{\rho_{1}}{t_{n}}\right)^{\min\{\frac{2 - \alpha}{r},\alpha + 1\}}, & n &\in [1, 2N] ,\\
%            &\left(\frac{\rho_{1}}{t_{n}}\right)^{\min\{\frac{2 - \alpha}{r},\alpha + 1\}}+\left(\frac{\rho_{1}}{t_{n}}\right)^{\alpha + 1} \lesssim \left(\frac{\rho_{1}}{t_{n}}\right)^{\min\{\frac{2 - \alpha}{r},\alpha + 1\}}, & n &\in [2N + 1, 4N], \\
%            &\left(\frac{\rho_{1}}{t_{n}^{(l)}}\right)^{\alpha + 1}+\left(\frac{\rho_{1}}{t_{n}}\right)^{\min\{\frac{2 - \alpha}{r},\alpha + 1\}} + \left(\frac{\rho_{1}}{t_{n}-\tau}\right)^{\min\{\frac{2 - \alpha}{r},\alpha + 2\}}(t_{n}-\tau), & n &\in [4N + 1, KN].
%        \end{aligned}
%    \end{cases}
%\end{equation*}
%\end{comment}
\end{proof}

Now we develop the fully discrete finite element scheme. Let $X_h$ be the continuous piecewise linear
finite element space in $H_0^1(\Omega)$
under a quasi-uniform mesh of $\Omega$
with a maximum diameter of $h$, and define
\(
(u,v)=\int_\Omega uvdx, ||u||=(u,v)^{\frac{1}{2}}.\)
Based on (\ref{eq:discretization of Caputo}), (\ref{eq:J825}) and the transformed form
(\ref{eq:After Integration}),
we propose the following fully discrete finite element scheme to solve the considered problem:
find $u_h^n\in X_h$ such that
\begin{equation}\label{eq:grade fully discrete scheme}
 (D^\alpha u_h^n,v_h)+B(u_h^n,v_h)=(bJ^{1-\alpha}u_h^{n-2N},v_h)+(G^n,v_h),\forall v_h\in X_h,
\end{equation}
where $u_h^{n-2N}=\varphi_h^{n-2N}\in X_h,1\leq n\leq 2N$ is a suitable approximation of $\varphi(x,t_{n-2N})$ and
\(B(u,v)=p(\frac{du}{dx},\frac{dv}{dx})-(au,v)\).
It is worth noting that the bilinear form satisfies the following well-known property
\begin{equation}\label{eq: B is a non-negative quantity}
    ||u||_{H_0^1(0,L)}^2 \lesssim  B(u,u).
\end{equation}

By observing \eqref{eq:250906a}, we note that when $r=1$,
the symmetric graded mesh \eqref{eq:250906a} will degenerate
into a uniform time mesh. In this work,
we will discuss the stability of the numerical scheme \eqref{eq:grade fully discrete scheme}
based on the uniform time mesh and the symmetric graded mesh,
and provide the local temporal error estimate for the case $r=1$ and the global temporal error estimate for the case $r>1$, respectively.

\section{Stability and convergence under uniform time mesh}
Based on the uniform time mesh, we discuss the stability of the fully discrete finite element scheme (\ref{eq:grade fully discrete scheme}) and provide its local time error estimate in this section.
Since the uniform grid is employed, the weight coefficient ‌$a^{(n)}_{n-k}$‌ defined in (\ref{eqs250829A}) is independent of the superscript $‌n$‌, and all $\rho_k$‌‌ values are identical. Therefore, we denote $\rho_k$‌ as ‌$\rho$‌ and ‌$a^{(n)}_{n-k}$‌ as ‌‌$a_{n-k}$‌ in this section.
%\subsection{Stability analysis}
Define $P_k,k=0,1,\cdots, n-1$ by \cite{Bu-2025-Local}
\begin{equation}\label{eq:8}
       P_0 = \frac{1}{a_0}, \quad P_{n-k} = \frac{1}{a_0} \sum_{j=k+1}^n P_{n-j} \left( a_{j-(k+1)} - a_{j-k} \right), \quad 1 \leq k \leq n-1.
\end{equation}
% (\ref{eq:8})式定义的$ P_l ,l=0,1\cdots $有如下两个性质：
%the $P_l, \, l = 0, 1, \cdots$ defined in (\ref{eq:8}) possess some useful properties that will be utilized later.
Now we introduce the following useful properties.
\begin{lemma}\cite{Bu-2025-Local,Dongfang}\label{lem:Properties of P}
For $P_k$ defined by (\ref{eq:8}), we have
\[\sum\limits_{j=1}^n P_{n-j} \leq \omega_{1+\alpha}(t_n),\]
\begin{equation*}
     \sum_{j=k}^n P_{n-j} a_{j-k} = 1, \quad 1 \leq k \leq n.
\end{equation*}
% \begin{equation*}
%      0  < P_i < \Gamma(2 - \alpha) \rho^\alpha (i + 1)^{\alpha - 1}, \quad 0 \leq i \leq n-1.
% \end{equation*}
What's more, for $2(i-1)N+1\leq n\leq 2iN$, $q=0,1,\cdots, i-1$ and $1\leq i \leq K$, it holds
\begin{equation*}
    \sum_{j=2qN+1}^n P_{n-j} (j - 2qN)^{-\beta} \lesssim \Gamma(2-\alpha) \rho^\alpha K_{\beta,n} (n - 2qN)^{\alpha - 1}, \quad \beta \geq 1,
\end{equation*}
% \begin{equation*}
%      \frac{1}{\Gamma(1-\beta)} \sum_{j=2qN+1}^n P_{n-j} (j - 2qN)^{-\beta} \lesssim \frac{\Gamma(2-\alpha) \rho^\alpha}{\Gamma(1-\beta + \alpha)} (n - 2qN)^{\alpha - \beta}, \quad \beta < 1,
% \end{equation*}
where
\(K_{\beta,n} =
\begin{cases}
1 + \dfrac{1 - n^{1 - \beta}}{\beta - 1}, & \beta \neq 1, \\
1 + \ln n, & \beta = 1,
\end{cases}
\)
is a strictly decreasing function about $\beta$ and one has $\lim\limits_{\beta\xrightarrow{}1}K_{\beta,n}=K_{1,n}$ for each n.
\end{lemma}

Let
\begin{equation*}
\mathcal{H} =
|b|\rho
\begin{pmatrix}
\mathbf{O}_{(n -2 N)\times 2N} & \mathbf{U}_{(n -2 N)\times (n - 2N)} \\
\mathbf{O}_{2N\times 2N} & \mathbf{O}_{2N\times(n -2 N)}
\end{pmatrix},
\end{equation*}
where $2(i-1)N \leq n \leq 2iN$, $ i =1,2,\cdots, K$, $\mathbf{O}_{I\times J}$ is a zero matrix with $I$ rows and $J$ columns, and
$\mathbf{U}_{(n-2N)\times (n-2N)}$ is an upper triangular matrix whose element $u_{jk}$ satisfies \(u_{jk}=1\) for $j\leq k$.
From the definition of $\mathcal{H}$, it is easy to check that
$\mathcal{H}^q = \mathbf{O}$ for all $q \geq i$. Now we utilize $\mathcal{H}$ to obtain an useful result, and then employ it to investigate the stability of the fully discrete finite element scheme (\ref{eq:grade fully discrete scheme}).

\begin{lemma}\label{lem:A-1}
Let $\mathcal{A} = [\
\underbrace{1, 1, \cdots,\ 1}_n\ ]^T$ with $2(i-1)N \leq n \leq 2iN,  i =2,3, \cdots, K $. Then the first element of $\mathcal{H}^q \mathcal{A}$ denoted as $\mathcal{E}_{n-2qN}^{(q)}, q=1,2,\cdots,i-1$ has the following estimates
\begin{equation}\label{eq:A_1}
\mathcal{E}_{n-2qN}^{(q)}> \mathcal{E}_{j}^{(q)}>0,j=1,2, \cdots, n-2qN-1\quad \text{and} \quad
    \mathcal{E}_{n-2qN}^{(q)}\lesssim 1.
\end{equation}
\end{lemma}

\begin{proof}
First we know that
\begin{equation*}
\begin{aligned}
    \mathcal{H}{\mathcal{A}}
&= |b|\rho
\left[
\begin{array}{ccccccc}
n - 2N , \  n - 2N - 1 ,  \ \cdots , 1 \ , 0 , \cdots , 0
\end{array}
\right]^T\\
&:=\left[
\begin{array}{cc}
     \mathcal{E}_{n-2N}^{(1)}, \, \mathcal{E}_{n-2N-1}^{(1)}, \ \cdots, \ \mathcal{E}_{1}^{(1)}, \ 0, \ 0, \ \cdots, \ 0   \\
\end{array}
\right]^T.
\end{aligned}
\end{equation*}
%其中\(|b|\rho (n-N)\lesssim C\).
It is clear that
%$\mathcal{E}_{l}^{(1)}=|b| \rho l \lesssim C,l=1,2, \cdots ,n-N$.
$\mathcal{E}_{n-2N}^{(1)}> \mathcal{E}_{j}^{(1)}>0,j=1,2,\cdots, n-2N-1 \ \text{and}\
\mathcal{E}_{n-2N}^{(1)}=|b| \rho (n - 2N) \lesssim 1$.
  Assume that (\ref{eq:A_1}) holds for $q=1,2,\cdots,r-1,  \ 2\leq r\leq i$, then
  %$\mathcal{E}_l^{(r-1)}\lesssim C,l=1,2,\cdots, n-(r-1)N,\ 2 \leq r \leq k$. Taking  $ q=r$, then
 \begin{equation*}
     \begin{aligned}
         \mathcal{H}^{r} \mathcal{A}
         &=\mathcal{H}\left[ \mathcal{E}_{n-2(r-1)N}^{(r-1)},  \ \mathcal{E}_{n-2(r-1)N-1}^{(r-1)} ,\cdots, \ \mathcal{E}_{1}^{(r-1)}, \ 0, \ 0, \ \cdots,0 \ \right]^T\\
        &=|b|\rho\left[ \sum_{j=1}^{n-2rN}\mathcal{E}_j^{(r-1)}, \ \sum_{j=1}^{n-2rN-1}\mathcal{E}_j^{(r-1)}, \  \cdots, \  \mathcal{E}_1^{(r-1)}, \ 0, \ 0, \cdots, \ 0 \    \right]^T\\
        &:=\left[ \mathcal{E}_{n-2rN}^{(r)}, \ \mathcal{E}_{n-2rN-1}^{(r)}, \ \cdots,  \mathcal{E}_{1}^{(r)},0 ,\  0, \ \cdots,0 \ \right]^T.
     \end{aligned}
 \end{equation*}
Thus, the first element of $\mathcal{H}^r \mathcal{A}$ satisfies
\begin{equation*}
\begin{aligned}
    \mathcal{E}_{n-2rN}^{(r)}
    &=|b|\rho \sum_{j=1}^{n-   2rN}\mathcal{E}_j^{(r-1)}\\
    &\lesssim |b|\rho (n-2rN)\mathcal{E}_{n-2rN}^{(r-1)}
    \lesssim 1,
\end{aligned}
\end{equation*}
and
\begin{equation*}
\begin{aligned}
\mathcal{E}_{n-2rN}^{(r)}&=|b|\rho \sum_{j=1}^{n-2rN}\mathcal{E}_j^{(r-1)}\\
&>|b|\rho \sum_{j=1}^{l}\mathcal{E}_{j}^{(r-1)}
=\mathcal{E}_{l}^{(r)}>0,
\end{aligned}
\end{equation*}
where $l=1,2, \cdots, n-2rN-1.$ Therefore the  mathematical induction implies that \eqref{eq:A_1} is correct.
%This completes the proof of the lemma.
\end{proof}

\begin{theorem}
    Let $u_h^n$ be the solution of (\ref{eq:grade fully discrete scheme}), it holds
\begin{equation} ||u_h^n||\lesssim \left(\left\| u_h^0 \right\| + \rho |b| \sum_{k=1}^{2N} \left\| \varphi_h^{k-2N} \right\|+ \frac{t_n^\alpha}{\Gamma(1+\alpha)} \max_{1 \leq j \leq n} \left\| G^j \right\|\right).
\end{equation}
\end{theorem}
\begin{proof}
Taking $v_h=u_h^j$ into (\ref{eq:grade fully discrete scheme}) yields
\begin{equation}\label{eq:eqs25721C}
    (D^\alpha u_h^j, u_h^j) + B(u_h^j, u_h^j) = b(J^{1-\alpha}u_h^{j-2N},u_h^j) + (G^j,u_h^j).
\end{equation}
According to (\ref{eq: B is a non-negative quantity})
and the fact \cite[Lemma 4.1]{Bu-2025-Local}
\begin{equation}\label{eqs250731B}
(D^{\alpha}u_{h}^{j},u_{h}^{j})\geq(D^{\alpha}\|u_{h}^{j}\|\ ) \|u_{h}^{j}\|,
\end{equation}
(\ref{eq:eqs25721C}) gives
\begin{equation*}
\begin{aligned}
    (D^\alpha ||u_h^j||)||u_h^j||&\leq
    b(J^{1-\alpha}u_h^{j-2N},u_h^j) + (G^j,u_h^j)\\
    &\leq
    |b|\cdot||J^{1-\alpha}u_h^{j-2N}||\cdot||u_h^j||+\|G^j\|\cdot||u_h^j||\\
    &\leq|b|J^{1 - \alpha}\|u_{h}^{j-2N}\|\|u_h^j\| + \|G^{j}\|\|u_h^j\|,
\end{aligned}
\end{equation*}
i.e.,
\begin{equation}\label{eqs250728A}
    \sum_{k = 1}^{j}a_{j - k}\nabla_{t}\|u_{h}^{k}\| \leq |b|\rho\sum_{k = 1}^{j}a_{j - k}\|u_{h}^{k-2N}\| + \|G^{j}\|.
\end{equation}
Multiplying both sides of (\ref{eqs250728A}) by \( P_{n-j} \),
and summing over \( j \) from \( 1 \) to \( n \), one has
\begin{equation}\label{eqs250728B}
        \sum_{j = 1}^{n}P_{n - j}\sum_{k = 1}^{j}a_{j - k}\nabla_{t}\|u_{h}^{k}\| \leq |b|\rho\sum_{j = 1}^{n}P_{n - j}\sum_{k = 1}^{j}a_{j - k}\|u_{h}^{k-2N}\| + \sum_{j = 1}^{n}P_{n - j}\|G^{j}\|.
    \end{equation}
Interchanging the summation order of $j$ and $k$ in the above inequality,
and  applying the second property of Lemma \ref{lem:Properties of P},
(\ref{eqs250728B}) can be derived as
\begin{equation}\label{eqs25721D}
 \begin{aligned}
    \left\| u_h^n \right\| &\leq \left\| u_h^0 \right\| + \rho |b| \sum_{k=1}^{2N} \left\| \varphi_h^{k-2N} \right\| + \rho |b| \sum_{k=2N+1}^{n} \left\| u_h^{k-2N} \right\|\ + \frac{t_n^\alpha}{\Gamma(1+\alpha)} \max_{1 \leq j \leq n} \left\| G^j \right\|\\
    &=z_n+\rho |b| \sum_{k=2N+1}^{n} \left\| u_h^{k-2N} \right\|,
 \end{aligned}
\end{equation}
where $z_n= \left\| u_h^0 \right\| + \rho |b| \sum\limits_{k=1}^{2N} \left\| \varphi_h^{k-2N} \right\|+ \frac{t_n^\alpha}{\Gamma(1+\alpha)} \max\limits_{1 \leq j \leq n} \left\| G^j \right\|$.

Let \( U = (||u_h^n||, ||u_h^{n - 1}||, \ldots,|| u_h^1||)^T \).
Then (\ref{eqs25721D}) can be rewritten as
\begin{equation}\label{eqs25721E}
U\leq\mathcal{H}U+z_n\mathcal{A}.
\end{equation}
By using (\ref{eqs25721E}) repeatedly and noting that $\mathcal{H}^i=\mathbf{O}$, one has
\begin{equation}\label{eqs250728C}
    \begin{aligned}
U
%&\leq \mathcal{H}U + z_n \mathcal{A} \\
&\leq \mathcal{H}(\mathcal{H}U + z_n \mathcal{A}) + z_n \mathcal{A} \\
&= \mathcal{H}^2 U + \mathcal{H} z_n \mathcal{A}+z_n\mathcal{A}\\
%&= \mathcal{H}^2 U + \sum_{q=0}^{1} (\mathcal{H})^q z_n \mathcal{A} \\
&\leq \cdots \\
&\leq z_n\sum_{q=1}^{i-1} \mathcal{H}^q  \mathcal{A}+z_n\mathcal{A}.
\end{aligned}
\end{equation}
 Since the first elements of $U, \mathcal{H}^q  \mathcal{A}$ are $||u_h^n||$ and $\mathcal{E}_{n-2qN}^{(q)},$ respectively. The combination of (\ref{eqs250728C}) and Lemma \ref{lem:A-1} immediately yields
\begin{equation}\label{eqs250909A}
    ||u_h^n||\lesssim \left(\left\| u_h^0 \right\| + \rho |b| \sum_{k=1}^{2N} \left\| \varphi_h^{k-2N} \right\|+ \frac{t_n^\alpha}{\Gamma(1+\alpha)} \max_{1 \leq j \leq n} \left\| G^j \right\|\right).
\end{equation}
\end{proof}

It is not difficult to observe that the above proof implies the following discrete Gronwall inequality holds.
\begin{lemma}\label{lem:Gronwall}
Assume that $\{y^n\}$ and $\{\varPhi^n\}$ be given nonnegative sequence satisfying
%Consider a non-negative real sequence $y^n$, $|b|$ is a non-negative constant.
\begin{equation}\label{eq:gronwall condition}
    y^n \leq |b| \rho \sum_{k = 2N + 1}^{n} y^{k - 2N} + \varPhi^n, \quad 2(i - 1)N +1\leq n \leq 2iN, \quad i=2,3,\cdots K.
\end{equation}
Then
$$y^n \lesssim \varPhi^n. $$
%where \(\varPhi=y^0+|b| t_N\mathop{max}\limits_{1\le k\le N}y^{k-N}\).
\end{lemma}
\begin{proof}
Let \(Y = \begin{bmatrix}y^n,y^{n - 1},\cdots,y^{n - 2N}\end{bmatrix}^T\).
   For \(2(i - 1)N+1\leq n\leq 2iN,i=2,3,\cdots, K\),
     (\ref{eq:gronwall condition}) can be rewritten as
\begin{equation}\label{eq:250731A}
    Y\leq \mathcal{H}Y + \varPhi^n\mathcal{A}.
\end{equation}
The subsequent procedure follows a pattern analogous to (\ref{eqs25721E})--(\ref{eqs250909A}), from which it is straightforward to derive
\[
y^n \lesssim \varPhi^n.
\]
The proof is completed.
\end{proof}

Define projection operator
 $R_h: H_0^1(\Omega) \xrightarrow{} X_h$ satisfying
\(B(R_hu,v)=B(u,v),\ \forall v\in X_h.\)
 When $u\in H_0^1(\Omega)\cap H^2(\Omega)$, it is obvious that \cite{Thomée2006Galerkin}
\begin{equation}\label{eq:h^2u}
    ||u-R_hu||\leq Ch^2||u||_{H^2(\Omega)}.
\end{equation}
With the help of the above lemma,
we can obtain the following error estimate.

\begin{theorem}\label{Them: Error results on the uniform grid}
Assume that $u$ be the solution of the
considered problem and satisfy the regularity (\ref{eqs250719A})--(\ref{eqs250719B}),
$u\in L^\infty(-\tau,K\tau;H^2(\Omega))$, $\partial_t u\in L^1(-\tau,K\tau;H^2(\Omega)),$
and $u^n_h, -2N\leq n\leq 0$ is a suitable approximation of $\varphi(x,t_n)$ such that
$\|u^n_h-\varphi(x,t_n)\|\lesssim h^2$. Then, under the uniform time mesh, the numerical solution $u_h^{n}$ satisfies
\begin{equation}\label{eq250906g}
 \left\lVert u^n - u^n_h \right\rVert\lesssim
\begin{cases}
    h^2  +\rho t_n^{\alpha-1}, \quad 1\leq n \leq 2N,\\
h^2 + \rho  ,\quad   2N+1\leq n\leq 2KN.
\end{cases}
\end{equation}
\end{theorem}
\begin{proof}
When \(2 N + 1 \leq n \leq 2KN \),
for $t = t_n$, (\ref{eq:After Integration}) gives
\begin{equation}\label{eq:semi-discrete}
    \left( {_0^C D_t^\alpha} u^n, v_h \right)+B(u^n,v_h) = b \left( {_0 I_t^{1 - \alpha}} u^{n - 2N}, v_h \right) + \left( G^n, v_h \right).
\end{equation}
By subtracting
\(\eqref{eq:grade fully discrete scheme}\) from \(\eqref{eq:semi-discrete}\),
one has
\begin{equation*}
    (D^\alpha (u^n-u_h^n),v_h)+B((u^n-u_h^n),v_h)=(bJ^{1-\alpha}(u^{n-2N}-u_h^{n-2N}),v_h)+((R_J^{1-\alpha})^{n-2N},v_h)- ((R_D^\alpha)^n,v_h),
\end{equation*}
where $(R_J^{1-\alpha})^{n-2N}:=(_0I_t^{1-\alpha}-J^{1-\alpha})u^{n-2N}\text{ and } (R_D^\alpha)^n:=({_0^C} D_t^\alpha-D^\alpha)u^n$.
Let $r_h^n=u^n-R_hu^n$ and $\varepsilon _h^n=R_hu^n-u_h^n$, the above equation can be written as
\begin{equation}\label{eqs250731C}
\begin{aligned}
(D^{\alpha} \varepsilon _h^n,v_h)+B(\varepsilon _h^n,v_h)=&
(b(R_J^{1-\alpha})^{n-2N},v_h)+(bJ^{1-\alpha}r_h^{n-2N},v_h)+(bJ^{1-\alpha}\varepsilon _h^{n-2N},v_h)\\
&-((R_D^\alpha)^n,v_h)-(D^\alpha r_h^n,v_h).
\end{aligned}
\end{equation}
Substituting $v_h=\varepsilon_h^n$
into (\ref{eqs250731C}), and using (\ref{eqs250731B}) and the Cauchy-Schwarz inequality, it yields
\begin{equation}\label{eqs250731D}
  \begin{aligned}
    D^\alpha||\varepsilon_h^n||&\leq |b|J^{1-\alpha}||\varepsilon_h^{n-2N}||+||(R_D^\alpha)^n||+|b|\cdot||(R_J^{1-\alpha})^{n-2N}||+|b|\cdot||J^{1-\alpha}r_h^{n-2N}||+||D^\alpha r_h^n||.
  \end{aligned}
  \end{equation}
Multiplying both sides of (\ref{eqs250731D}) by \( P_{n-j} \)
and summing over \( j \) from \( 1 \) to \( n \), it can be obtained that
\begin{equation}\label{eq250906f}
\begin{aligned}
    \sum_{j = 1}^{n}P_{n - j}\sum_{k = 1}^{j} a_{j - k}\nabla _t||\varepsilon_h^k ||\leq &|b|\rho\sum_{j = 1}^{n}P_{n - j}\sum_{k = 1}^{j} a_{j - k}||\varepsilon_h^{k - 2N}||+ |b|\sum_{j = 1}^{n} P_{n - j}  ||J^{1-\alpha}r_h^{j-2N}|| + \sum_{j = 1}^{n}P_{n - j}\|(R_D^{\alpha})^j\|\\
    &+\sum_{j = 1}^{n}P_{n - j}||D^\alpha r_h^j|| + |b|\sum_{j = 1}^{n}P_{n - j}\|(R_J^{1-\alpha})^{j - 2N}\|.
\end{aligned}
\end{equation}
Employing a derivation analogous to that from (\ref{eqs250728B}) to (\ref{eqs25721D}),
(\ref{eq250906f}) leads to
\begin{equation}\label{eqs250910A}
\begin{aligned}
   ||\varepsilon_h^n ||\leq &||\varepsilon_h^0||+ |b|\rho\sum_{k = 1}^{2N}||\varepsilon_h^{k - 2N}||+ |b|\rho\sum_{k = 2N + 1}^{n}||\varepsilon_h^{k - 2N}||
   + |b| \sum_{j = 1}^{n} P_{n - j} \|J^{1- \alpha} r_h^{j - 2N}\|\\
   & +\sum_{j = 1}^{n}P_{n - j}\|(R_D^{\alpha})^j\|+\sum_{j = 1}^{n} P_{n - j} \|D^{\alpha} r_h^{j}\| + |b|\sum_{j = 1}^{n}P_{n - j}\|(R_J^{1-\alpha})^{j - 2N}\|.
\end{aligned}
\end{equation}

Let
\[
\begin{aligned}
    \Phi^n=&\left( |b|\sum_{j = 1}^{n}P_{n - j}\|(R_J^{1-\alpha})^{j - 2N}\|+\sum_{j = 1}^{n}P_{n - j}\|(R_D^{\alpha})^j\| \right)\\
    &+\left(||\varepsilon_h^0||+ |b|\rho\sum_{k = 1}^{2N}||\varepsilon_h^{k - 2N}||+
   |b| \sum_{j = 1}^{n} P_{n - j} \|J^{1- \alpha} r_h^{j - 2N}\|+\sum_{j = 1}^{n} P_{n - j} \|D^{\alpha} r_h^{j}\|\right).
\end{aligned}
\]
Therefore (\ref{eqs250910A}) gives
\begin{equation}
\begin{aligned}
    ||\varepsilon_h^n ||\leq & |b|\rho\sum_{k = 2N + 1}^{n}||\varepsilon_h^{k - 2N}||+\Phi^n.
\end{aligned}
\end{equation}
According to Lemmas \ref{lem: Truncation error of J on the hierarchical grid}, \ref{lem: Truncation error of D on the hierarchical grid} and \ref{lem:Properties of P}, for $r=1$ it implies that
\begin{equation}\label{eqs250910B}
\begin{aligned}
    \sum_{j = 1}^{n} P_{n - j} \| (R_{J}^{1 - \alpha})^{j - 2N}\| &\lesssim \sum_{j = 1}^{n} P_{n - j} \rho
    \lesssim \rho,
\end{aligned}
\end{equation}
and
\begin{equation}\label{eqs250910C}
\begin{aligned}
    \sum_{j = 1}^{n} P_{n - j} \| (R_{t}^\alpha)^j \| &\lesssim \sum_{j = 1}^{2N} P_{n - j} \left(\frac{\rho}{t_j}\right)^{\min\{2 - \alpha, 1 + \alpha\}} + \sum_{j = 2N + 1}^{n} P_{n - j}  \left[\rho^{1+\alpha}(t_j-\tau)^{-\alpha} +
    \rho^{2 - \alpha} (t_j - \tau)^{\alpha - 1}\right] \\
    &= \sum_{j = 1}^{2N} P_{n - j} \left(\frac{\rho}{t_j}\right)^{\min\{2 - \alpha, 1 + \alpha\}} + \sum_{j = 2N + 1}^{n} P_{n - j}\left[\rho \left( \frac{\rho}{t_j - \tau} \right)^\alpha + \rho \left( \frac{\rho}{t_j - \tau} \right)^{1 - \alpha}\right]
    \\
    &\lesssim \sum_{j = 1}^{n} P_{n - j} j^{-\min\{2-\alpha, 1 + \alpha\}} +  \sum_{j = 2N + 1}^{n} P_{n-j} \rho  \\
    &\lesssim \rho^\alpha K_{\beta_1,n}n^{\alpha-1} +\rho\\
    &\lesssim \rho
\end{aligned}
\end{equation}
where $\beta_1={\min\{2-\alpha, 1 + \alpha\}}$.
From
(\ref{eq:discretization of Caputo}), (\ref{eq:J825}), (\ref{eq:h^2u}), Lemma \ref{lem:Properties of P} and the assumptions of Theorem \ref{Them: Error results on the uniform grid}, one has
\begin{equation}\label{eqs250910D}
||\varepsilon_h^0||+ |b|\rho\sum_{k = 1}^{2N}||\varepsilon_h^{k - 2N}||\lesssim h^2,
\end{equation}
\begin{equation}\label{eqs250901K}
 \begin{aligned}
&\sum_{j = 1}^{n} P_{n - j} \|J^{1- \alpha} r_h^{j - 2N}\|\\
\lesssim &\rho \sum_{k = 1}^{n} \left(\sum_{j = k}^{n} P_{n - j} a_{j - k} \right) \|(u - R_h u)(t_{k - 2N})\|\\
\lesssim& h^2,
\end{aligned}
\end{equation}
and
\begin{equation}\label{eqs250901J}
 \begin{aligned}
\sum_{j = 1}^{n} P_{n - j} \|D^{\alpha} r_h^{j}\| %&= \sum_{j = 1}^{n} P_{n - j} \left\|\sum_{k = 1}^{j} a_{j - k} \nabla_t (u - R_h u)(t_k)\right\|\\
&\lesssim \sum_{j = 1}^{n} P_{n - j} \sum_{k = 1}^{j} a_{j - k} \int_{t_{k - 1}}^{t_k} \|\partial_t (u - R_h u)\| dt\\
&= \sum_{k = 1}^{n} \left(\sum_{j = k}^{n} P_{n - j} a_{j - k} \right) \int_{t_{k - 1}}^{t_k} \|\partial_t (u - R_h u)\| dt\\
&\lesssim  h^2 \int_{0}^{t_n} \|\partial_t u\|_{H^2(\Omega)} dt\\
&\lesssim h^2.
\end{aligned}
\end{equation}
The estimations (\ref{eqs250910B})--(\ref{eqs250901J}) demonstrate $\Phi^n\lesssim\rho+h^2.$
Thus it follows from Lemma \ref{lem:Gronwall} that
\begin{align}\label{eq:v1}
    || \varepsilon _h^n||\lesssim h^2+\rho.
\end{align}
When \(1 \leq n \leq 2N\), following the same derivation method, we still can obtain error
equation (\ref{eqs250910A}), with the sole requirement that $\Phi^n$ reduces to
\[
\begin{aligned}
    \Phi^n=&\left( |b|\sum_{j = 1}^{n}P_{n - j}\|(R_J^{1-\alpha})^{j - 2N}\|+\sum_{j = 1}^{n}P_{n - j}\|(R_D^{\alpha})^j\| \right)\\
    &+\left(||\varepsilon_h^0||+
   |b| \sum_{j = 1}^{n} P_{n - j} \|J^{1- \alpha} r_h^{j - 2N}\|+\sum_{j = 1}^{n} P_{n - j} \|D^{\alpha} r_h^{j}\|\right).
\end{aligned}
\]
Similarly, as discussed above, it is clear that $\Phi^n\lesssim h^2+\rho.$ It means
\begin{equation}\label{eqs250910E}
||\varepsilon_h^n||\lesssim h^2+\rho.
\end{equation}
Lastly, due to \( \|u^n-u_h^n\| \leq| |r_h^n||+|| \varepsilon  _h^n||\), the combination of \eqref{eq:h^2u}, \eqref{eq:v1} and \eqref{eqs250910E},
it can be concluded that (\ref {eq250906g}) holds.
\end{proof}

\section{Stability and convergence under non-uniform time mesh}
In this section, we analyze the stability of the numerical scheme (\ref{eq:grade fully discrete scheme}) and its convergence based on the symmetric graded time mesh (\ref{eq:250906a}) with $r>1$, providing globally optimal time error estimate.
First, we introduce the following classical discrete fractional Gronwall inequality.
\begin{lemma}\cite{Bu-2024-Finite,Liao-Sharp-2018}\label{lem:gronwall}
For given nonnegative sequences \(\{y_{j}, z_{j} \mid 1 \leq j \leq 2KN\}\) and \(\{\lambda_{j} \mid 0 \leq j \leq 2KN-1\}\). Suppose that there exists a constant \(\Lambda\) (independent of the step size) with \(\Lambda \geq \sum\limits_{j=0}^{2KN-1} \lambda_{j}\), and
$
\max_{1 \leq n \leq 2KN} \rho_{n} \leq \frac{1}{\sqrt[\alpha]{2\Gamma(2-\alpha)\Lambda}}.
$
If the grid function $\{v_{k} \mid k \geq 0\}$ satisfies
\begin{equation}
    \sum_{k=1}^{n} a_{n-k}^{(n)} \nabla_{t} (v_{k})^{2} \leq \sum_{k=1}^{n} \lambda_{n-k} (v_{k})^{2} + v_{n}(y_{n} + z_{n}), \quad 1 \leq n \leq 2KN,
\end{equation}
then
\[
v_{n} \leq 2E_{\alpha,1}\left(2\max(1,\sigma)\Lambda t_{n}^{\alpha}\right) \left(v_{0} + \max_{1 \leq k \leq n} \sum_{j=1}^{k} \overline{P}_{k-j}^{(k)} y_{j} + \omega_{1+\alpha}(t_{n}) \max_{1 \leq j \leq n} z_{j}\right),
\]
where $E_{\alpha,1}(\cdot) := \sum\limits_{i=0}^{\infty} \frac{(\cdot)^{i}}{\Gamma(i\alpha + 1)}, \sigma=max\{3^{r-1},2^r-1\}$ and
\begin{equation}\label{eq:Definition of P-GradingA}
\overline{P}_{k-j}^{(k)} :=
\begin{cases}
\dfrac{1}{a_{0}^{(k)}}, \quad j = k, \\
\dfrac{1}{a_{0}^{(j)}} \sum\limits_{i=j+1}^{k} \overline{P}_{k-i}^{(k)} \left(a_{i-(j+1)}^{(i)} - a_{i-j}^{(i)}\right), \quad 1 \leq j \leq k-1.
\end{cases}
\end{equation}
\end{lemma}

\begin{remark}
Here, \(\{\overline{P}_{k-j}^{(k)}\}\) is defined in (\ref{eq:Definition of P-GradingA})
via the discretization coefficient $a_{k-j}^{(k)}$.
Since under the symmetric graded time mesh, each value of
$a_{k-j}^{(k)}$ depends on the superscript $k$,
it is evident that
\(\{\overline{P}_{k-j}^{(k)}\}\) differs from
$\{P_{k}\}$ determined by (\ref{eq:8}).
\end{remark}

Now we state the stability result of the fully discrete finite element scheme (\ref{eq:grade fully discrete scheme}).
\begin{theorem}
    Let $u_h^n$
  be the solution of the fully discrete scheme \eqref{eq:grade fully discrete scheme}. Then
\begin{equation}\label{eqs250907b}
    ||u_h^n|| \leq 2E_{\alpha,1}\left(2\max(1,\sigma)\Lambda t_{n}^{\alpha}\right) \left(||u_h^{0}|| + \omega_{1+\alpha}(t_{n}) \max_{1 \leq j \leq n} z_{j}\right),
\end{equation}
where $ z_j=2(c|b| \omega_{2-\alpha}(t_j)  + \| G^j \|)$,
$c=\sup\limits_{1 \leq k \leq 2N}||u^{k-2N}||$ and
\(\Lambda=(|b|^2+1)\omega_{2-\alpha}(K\tau)\).

\end{theorem}
\begin{proof}
Taking $v_h=2u_h^n$ into \eqref{eq:grade fully discrete scheme} yields
\begin{equation}\label{eqs250825F}
\begin{aligned}
&(D^\alpha u_h^n, 2u_h^n) + B(u_h^n, 2u_h^n) = b\left( \sum_{k = 1}^n \rho_k a_{n - k}^{(n)} u_h^{k - 2N}, 2u_h^n \right) + (G^n, 2u_h^n).
\end{aligned}
  \end{equation}
From \cite[Lemma 1]{Anatoly-scheme-2015424}, we have
\begin{equation}\label{eqs250829F}
      \begin{aligned}
(D^\alpha u_h^n, 2u_h^n) \geq \sum_{k = 1}^n a_{n - k}^{(n)} \nabla_t (u_h^k)^2.
\end{aligned}
  \end{equation}
  The combination of \eqref{eq: B is a non-negative quantity},\eqref{eqs250825F} and \eqref{eqs250829F}, it can be obtained that
\begin{equation}\label{eqs250829B}
      \begin{aligned}
 \sum_{k = 1}^n a_{n - k}^{(n)} \nabla_t (u_h^k)^2 \leq 2 \left[ |b|\sum_{k = 1}^n \rho_k a_{n - k}^{(n)} ||u_h^{k - 2N}|| + \| G^n \| \right] \| u_h^n \|,
\end{aligned}
  \end{equation}
where the Cauchy-Schwarz inequality is used.
By observing (\ref{eqs250829A}), it is clear that
\begin{equation}\label{eqs250828C}
\begin{aligned}
    \sum_{k=1}^n\rho_ka_{n-k}^{(n)} &=\sum_{k=1}^n \left(\omega_{2-\alpha}(t_{n}-t_{k-1}) - \omega_{2-\alpha}(t_{n}-t_{k})\right)\\
    &=\omega_{2-\alpha}(t_n).
\end{aligned}
\end{equation}
Hence, it follows from \eqref{eqs250829B} and \eqref{eqs250828C} that
\begin{equation}\label{eqs250829D}
      \begin{aligned}
 \sum_{k = 1}^n a_{n - k}^{(n)} \nabla_t (u_h^k)^2 \leq 2 \left (c|b|\omega_{2-\alpha}(t_n) + \| G^n \| \right) \| u_h^n \|,
\end{aligned}
\end{equation}
for $1 \leq n \leq 2N$; and
\begin{equation}\label{eqs250819E}
      \begin{aligned}
 \sum_{k = 1}^n a_{n - k}^{(n)} \nabla_t (u_h^k)^2
 &\leq 2 \left[ |b|\sum_{k = 1}^{2N} \rho_k a_{n - k}^{(n)} ||u_h^{k - 2N}|| +|b|\sum_{k = 2N+1}^{n} \rho_k a_{n - k}^{(n)} ||u_h^{k - 2N}||+  \| G^n \| \right] \| u_h^n \|\\
% &\leq 2|b| \sum_{k = 2N+1}^{n} \rho_k a_{n - k}^{(n)} ||u_h^{k - 2N}|| \cdot ||u_h^n||+2(c |b|\omega_{2-\alpha}(\tau)  + \| G^n \|)\\
 &\leq  \sum_{k = 2N+1}^{n} \rho_k a_{n - k}^{(n)}\left(|b|^2 ||u_h^{k - 2N}||^2+||u_h^n||^2 \right)
 +2(c |b|\omega_{2-\alpha}(t_n)  + \| G^n \|)||u_h^n||\\
 %&\leq |b|^2\sum_{k = 2N+1}^{n} \rho_k a_{n - k}^{(n)}||u_n^{k-2N}||^2+\omega_{2-\alpha}(t_n-\tau)||u_h^n||^2 + 2(c |b|\omega_{2-\alpha}(\tau)  + \| G^n \|)||u_h^n||\\
 &\leq |b|^2\sum_{k = 1}^{n-2N} \rho_k a_{n -2N- k}^{(n)}||u_n^{k}||^2
 +\omega_{2-\alpha}(t_n)||u_h^n||^2 + 2(c |b|\omega_{2-\alpha}(t_n)  + \| G^n \|)||u_h^n||,
\end{aligned}
  \end{equation}
for $2N+1 \leq n \leq 2KN$.
Based on (\ref{eqs250829D}) and (\ref{eqs250819E}), it is obvious that for $1\leq n \leq 2KN$, we have the following uniform result
\begin{equation}\label{eqs250907a}
\sum_{k = 1}^n a_{n - k}^{(n)} \nabla_t (u_h^k)^2
\leq |b|^2\sum_{k = 1}^{n-2N} \rho_k a_{n -2N- k}^{(n)}||u_n^{k}||^2
+\omega_{2-\alpha}(t_n)||u_h^n||^2 + 2(c |b|\omega_{2-\alpha}(t_n)  + \| G^n \|)||u_h^n||.
\end{equation}
Lastly, applying Lemma \ref{lem:gronwall} to \eqref{eqs250907a}, it yields
\eqref{eqs250907b} immediately.
\end{proof}

Now we consider the convergence of the fully discrete finite element scheme \ref{eq:grade fully discrete scheme}) based on the symmetric graded time mesh.
By applying a derivation process analogous to that used for obtaining \eqref{eqs250731C},
we can easily obtain the error equation for numerical scheme \ref{eq:grade fully discrete scheme}) under the non-uniform time mesh.
Therefore, for the sake of brevity, we omit the derivation and directly state the result
\begin{equation}\label{eqs250829G}
\begin{aligned}
({D}^{\alpha}\varepsilon_{h}^{n},v_h)+B(\varepsilon_{h}^{n},v_h)&=(bJ^{1 - \alpha}\varepsilon_{h}^{n - 2N},v_h)+(b(R_J^{1-\alpha})^{n-2N},v_h)+(bJ^{1 - \alpha}r_{h}^{n - 2N},v_h)\\
&\quad-((R_D^{\alpha})^n,v_h)-({D}^{\alpha}r_{h}^{n},v_h)\\
&:=(bJ^{1 - \alpha}\varepsilon_{h}^{n - 2N},v_h)+(R^{n},v_h),
    \end{aligned}
\end{equation}
where $\varepsilon_{h}^{n}, r_{h}^{n}$ are defined in Section 3, and
\begin{equation}\label{eqs250911def}
\begin{aligned}
    R^{n}=&-{D}^{\alpha}r_{h}^{n}+bJ^{1 - \alpha}r_{h}^{n - 2N}-(R_D^{\alpha})^{n}+b(R_J^{1-\alpha})^{n-2N}\\
    :=&R^{n}_{1}+R^{n}_{2}+R^{n}_{3}+R^{n}_{4}.
\end{aligned}
\end{equation}
Taking \(v_h = 2\varepsilon_h^n\) into \eqref{eqs250829G}, and using \eqref{eq: B is a non-negative quantity}, \eqref{eqs250829F} and the Cauchy-Schwarz inequality leads to
\begin{equation}\label{eqs250901H}
\begin{aligned}
\sum_{k=1}^{n}a_{n-k}^{(n)}(\varepsilon_h^n)^2\leq& 2|b|\sum_{k=1}^{n}\rho_k a_{n-k}^{(n)}||\varepsilon_h^{k-2N}||\cdot ||\varepsilon_h^{n}||+
2||R^n||\cdot||\varepsilon_h^{n}||\\
\leq&|b|^2\sum_{k = 1}^{n-2N} \rho_k a_{n -2N- k}^{(n)}||\varepsilon_n^{k}||^2
+\omega_{2-\alpha}(t_n)||\varepsilon_h^n||^2 + 2( |b|\omega_{2-\alpha}(t_n)\max\limits_{1\leq k\leq 2N}\|\varepsilon_h^{k-2N}\|  + \| R^n \|)||\varepsilon_h^n||.
\end{aligned}
\end{equation}
What is more, it follows from
Lemma \ref{lem:gronwall} to \eqref{eqs250901H}, \eqref{eqs250901H} can be derived as
\begin{equation}\label{eq:really gronwall}
\left\|\varepsilon_{h}^{n}\right\| \leq 2E_{\alpha,1}\left(2\max(1,\sigma)|\Lambda|(t_{n})^{\alpha}\right)\left(\|\varepsilon_{h}^{0}\|+\max_{1\leq k\leq n}\sum_{j = 1}^{k}\overline{P}_{k - j}^{(k)}y_{j}\right),
\end{equation}
where \(y_{j} = 2\left( |b|\omega_{2 - \alpha}(t_j)\max\limits_{1 \leq k \leq 2N} \| \varepsilon_h^{k-2N} \| + \left\|R^{j}\right\| \right)\), and \(
\Lambda = \omega_{2 - \alpha}(K\tau)(|b|^{2} + 1)\).

In order to investigate the global time error estimate of (\ref{eq:grade fully discrete scheme}), we introduce two useful lemmas firstly.

\begin{lemma}\cite{Tan}\label{lem:Properties of P on a Hierarchical Grid}
For the discrete coefficient \( \overline{P}_{n - k}^{(n)} \) defined by \eqref{eq:Definition of P-GradingA}, one has
\begin{align}
    \sum_{j = 2(k-1)N + 1}^{n} \overline{P}_{n - j}^{(n)} \omega_{1-\alpha}\left(t_j - (k - 1)\tau\right) &\leq 1, \quad 1 \leq k \leq K.
\end{align}
\end{lemma}
\begin{lemma}\label{lem:the four}
Assume that $u$ be the solution of the
considered problem and satisfy the regularity (\ref{eqs250719A})--(\ref{eqs250719B}),
$u\in L^\infty(-\tau,K\tau;H^2(\Omega))$, $\partial_t u\in L^1(-\tau,K\tau;H^2(\Omega)),$
and $u^n_h, -2N\leq n\leq 0$ is a suitable approximation of $\varphi(x,t_n)$ such that
$\|u^n_h-\varphi(x,t_n)\|\lesssim h^2$.
Then, for \(R^n_i, i=1,2,3,4\) mentioned in \eqref{eqs250911def}, the following results hold
\begin{align*}
\max_{1 \leq i \leq n} \sum_{j=1}^{i} \overline{P}_{i-j}^{(i)} \| R^j_1 \| &\lesssim h^2 , \quad
\max_{1 \leq i \leq n} \sum_{j=1}^{i} \overline{P}_{i-j}^{(i)} \| R^j_2 \| \lesssim h^2, \\
\max_{1 \leq i \leq n} \sum_{j=1}^{i} \overline{P}_{i-j}^{(i)} \| R^j_3 \| &\lesssim N^{-\min\{2 - \alpha,\alpha r\}}, \quad
\max_{1 \leq i \leq n} \sum_{j=1}^{i} \overline{P}_{i-j}^{(i)} \| R^j_4 \| \lesssim N^{-1}.
\end{align*}
\end{lemma}
\begin{proof}
First, by employing the  derivations analogous to \eqref{eqs250901K} and \eqref{eqs250901J}, it is easy to check that
\[\max_{1 \leq i \leq n} \sum_{j=1}^{i} \overline{P}_{i-j}^{(i)} \| R^j_1 \| \lesssim h^2 , \quad
\max_{1 \leq i \leq n} \sum_{j=1}^{i} \overline{P}_{i-j}^{(i)} \| R^j_2 \| \lesssim h^2.\]
Due to $\rho_1^{1+\alpha}(t_j-\tau)^{-\alpha}\lesssim\rho_1^{1/r}$ and
$\rho_1^{(2-\alpha)/r}(t_j-\tau)^{(\alpha-2)/r+1}=\rho_1^{1/r}\left(\frac{\rho_1}{t_j-\tau}\right)^{(1-\alpha)/r}(t_j-\tau)^{1-1/r}\lesssim\rho_1^{1/r}$
for $j\geq2N+1,r>1$,
from Lemma \ref{lem: Truncation error of D on the hierarchical grid} and Lemma \ref{lem:Properties of P on a Hierarchical Grid}, one has
\begin{equation}\label{eq:R3}
    \begin{aligned}
     \max_{\substack{1 \leq i \leq n}} \sum_{j=1}^{i}\overline{P}_{i-j}^{(i)} \| R^j_3 \|
=& \max_{\substack{1 \leq i \leq n}} \left( \sum_{j=1}^{\min\{i,2N\}} \overline{P}_{i-j}^{(i)} \| R^j_3 \| + \sum_{j=2N+1}^{i} \overline{P}_{i-j}^{(i)} \| R^j_3 \| \right) \\
\lesssim& \max_{\substack{1 \leq i \leq n}} \Bigg( \sum_{j=1}^{\min\{i,2N\}} \overline{P}_{i-j}^{(i)}\omega_{1-\alpha}(t_j)t_j^\alpha\left( \frac{\rho_1}{t_j}\right)^{\min\{\frac{2-\alpha}{r},\alpha+1\}}\\
&+\sum_{j=2N+1}^{i} \overline{P}_{i-j}^{(i)} \omega_{1-\alpha}(t_j-\tau) (t_j-\tau)^\alpha\rho_1^{\frac{1}{r}}\Bigg) \\
\lesssim&  \max_{\substack{1 \leq i \leq n}} \left(\max_{\substack{1 \leq j \leq 2N}} \left\{t_j^\alpha \left( \frac{\rho_1}{t_j}\right)^{\min\{\frac{2-\alpha}{r},\alpha+1\}}\right\} +
\max_{\substack{2N+1 \leq j \leq i}}\left\{(t_j - \tau)^\alpha   \rho_1^{\frac{1}{r}}\right\} \right),
\end{aligned}
\end{equation}
where we stipulate that $\sum_{j=l}^m w_j=0$ and $\max_{l\leq j\leq m}w_j=0$ for $l>m.$
In \cite[The proof of Theorem 5.1]{Bu-2024-Finite}, by discussing three
different cases i.e. $\alpha<\frac{2 - \alpha}{r}- 1, \frac{2 - \alpha}{r}-1\leq \alpha<\frac{2 - \alpha}{r}$ and $\alpha>\frac{2 - \alpha}{r}- 1$,
it has been shown that
\[
t_{j}^{\alpha}\left(\frac{\rho_{1}}{t_{j}}\right)^{\frac{2 - \alpha}{r}}
\lesssim \rho_{1}^{\min\{\frac{2 - \alpha}{r},\alpha\}}.
\]
Therefore \eqref{eq:R3} can be bounded by
\begin{equation*}
    \begin{aligned}
      \max_{1\leq k\leq n}\sum_{j = 1}^{k}P_{k - j}^{(k)}\left\|R_{j}^{3}\right\|
      &\lesssim \rho_{1}^{\min\{\frac{2 - \alpha}{r},\alpha\}}+ \rho_1^{\frac{1}{r}}\\
      &\lesssim N^{-\min\{1,\alpha r\}}.
    \end{aligned}
\end{equation*}
%\(R_{n}^{4}=b(R_J^{1-\alpha})^{n-N}\),
Besides, based on Lemma \ref{lem: Truncation error of J on the hierarchical grid}
and Lemma \ref{lem:Properties of P on a Hierarchical Grid}, it is clear that
\begin{equation*}\label{eq:Rj4}
    \begin{aligned}
        \max_{1\leq i\leq n}\sum_{j = 1}^{i}P_{i - j}^{(i)}\left\|R_{j}^{4}\right\|&\lesssim\max_{1\leq i\leq n}\left(\sum_{j = 1}^{i}\overline{P}_{i - j}^{(i)}\omega_{1-\alpha}(t_j)t_j^\alpha\rho_1^{\frac{1}{r}}\right)\\
        &\lesssim \rho_1^{\frac{1}{r}}
        \lesssim N^{-1}.
    \end{aligned}
\end{equation*}
The proof is completed.
\end{proof}

Now applying Lemma \ref{lem:the four} to \eqref{eq:really gronwall}, we can conclude that
\begin{equation}\label{eqs250814A}
\begin{aligned}
\left\|\varepsilon_{h}^{n}\right\|
    &\leq 2E_{\alpha,1}\left(2\max(1,\sigma)|\Lambda|t_{n}^{\alpha}\right)
    (h^2
    + N^{-\min\{1,\alpha r\}}
    + N^{-1})\\
    &\lesssim h^2+N^{-\min\{1,\alpha r\}}.
\end{aligned}
\end{equation}
Thus, in virtue of (\ref{eq:h^2u}), (\ref{eqs250814A}) and the fact $\left\|u^{n}-u_{h}^{n}\right\|\leq\left\|\varepsilon_{h}^{n}\right\|+\left\|r_{h}^{n}\right\| $, the following global time convergence result can be obtained.
\begin{theorem}
Under the assumptions of Lemma \ref{lem:the four}, the numerical solution \(u_h^{n}\)
from (\ref{eq:grade fully discrete scheme}) based on the symmetric graded time mesh
satisfies
$$ \left\|u^{n}-u_{h}^{n}\right\|\lesssim h^{2}+N^{-\min\{1,\alpha r\}}. $$
\end{theorem}

\section{Numerical experiment}
In this section, we validate the established theoretical results by several numerical tests
under two different cases. In the first case, an exact solution of the considered problem \eqref{eq:main-equation}--\eqref{eq:boundary conditions} which has the regularities \eqref{eqs250719A}--\eqref{eqs250719B} is provided, while the second case gives the
right-hand side function and initial function satisfying the conditions mentioned in \cite[Theorem 1]{Bu-2025-Finite} to make sure that the regularity assumptions \eqref{eqs250719A}--\eqref{eqs250719B} are satisfied. In order to examine the errors
and convergence orders, we define
\[E(M,N,k,l)=\max_{2kN+1 \leq n \leq 2lN} \| \bar{u}^n - u_h^{n} \|,\]
and
\[rate_t=\log_2\left(\frac{E(M,N,k,l)}{E(M,2N,k,l)}\right), \
rate_s=\log_2\left(\frac{E(M,N,k,l)}{E(2M,N,k,l)}\right),\]
where $M$ denotes the number of spatial subdivisions, $\bar{u}^n = u^n$ when the exact solution is known, otherwise $\bar{u}^n$ represents an approximation of $u^n$ on a sufficiently fine mesh.
%and the \(rate_t\) is employed to assess the local temporal convergence over the intervals \((k-1,k]\) for \(k = 1,2,3\), while the \(rate_s\) is used to study the spatial case convergence, for the uniform grid, \(rate_t\) denotes the convergence order related to k. for the symmetric grade mesh, \(rate_t\) and \(rate_s\) denotes that it is independent of k. \\

\textbf{Example 1.} For the considered problem \eqref{eq:main-equation}--\eqref{eq:boundary conditions}, assume that $\tau=b=1$,
the spatial domain is \([0, 1]\) and the temporal interval is \((0, 3]\). Two different
cases are considered as follows:

$\bullet$ In the first case,
let \(p=\frac{1}{\pi^2}\), \(a=-2\)  and the initial function \( \varphi(x,t) =  (1+t)\sin(\pi x) \). Here, we suppose that the exact solution is
\[
u(x,t) =sin(\pi x )
\begin{cases}
1 + t + t^\alpha, & 0 < t \leq 1, \\
1 + t + t^\alpha + (t - 1)^{\alpha + 1}, & 1 < t \leq 2 ,\\
1 + t + t^\alpha + (t - 1)^{\alpha + 1} + (t - 2)^{\alpha + 2}, & 2<t \leq 3.
\end{cases}
\]
By combining the known conditions presented above, it is clear that we can determine the corresponding right-hand side function $f(x,t)$ easily.

$\bullet$ In the second case, we take $p = \frac{1}{5}, a = -1,$ \(f(x,t) = t^2\sin(\pi x)\)
and $\varphi(x,t) = (1+\pi t)\sin(\pi x)$.

We show the numerical results in Tables \ref{tab:numerical_temporal_accuracy}--\ref{tab:space 1/alp}.
Tables \ref{tab:numerical_temporal_accuracy}--\ref{tab:exact spical grade}
are obtained based on the the first case of Example 1, while Tables \ref{tab:The case of uniform grid without true solution}--\ref{tab:space 1/alp} are obtained under the second
case. In addition, when we compute Example 1 based on the uniform time mesh, the numerical results are presented in Tables \ref{tab:numerical_temporal_accuracy}, \ref{tab:spatial2}, \ref{tab:The case of uniform grid without true solution} and \ref{tab:space without exact solution}, and the remaining tables provide the numerical results under the symmetric graded time mesh.
By fixing $M=1000$ to eliminate the influence of spatial errors, and taking $\alpha=0.5, 0.7$ and $\alpha=0.6, 0.8$ respectively, Tables \ref{tab:numerical_temporal_accuracy} and \ref{tab:The case of uniform grid without true solution} demonstrate that the error $E(M,N,k,l)$ exhibits $\alpha$-order convergence on the interval $(0, 1]$, while it reaches first-order convergence on $(1, 3]$.
In Tables \ref{tab:addlabel} and \ref{tab:space without exact solution}, we choose $N=5000$ with different $\alpha$ to examine the spatial convergence rate at $t=t_{6N}$. Evidently, it implies that the spatial convergence order is $2$.
Now we observe the numerical results which are computed based on the symmetric graded
time mesh. Similarly, we still take $M=1000$ to test the temporal convergence rate and choose $N=5000$ to examine the spatial accuracy.
Tables \ref{tab:alp=0.5,r=3/5}, \ref{tab:alp=0.7,r=3/5}, and \ref{tab:The case of symmetric grade grid without true solution} are used to examine the temporal convergence rate
with $\alpha = 0.5, 0.7$ and $0.6$. It indicates that the convergence order is $\min\{1, r\alpha\}$. In Tables \ref{tab:exact spical grade} and \ref{tab:space 1/alp}, we choose $r = \frac{1}{\alpha}$ to check the spatial accuracy which shows a second-order convergence rate. It is obvious that all numerical results in Tables \ref{tab:numerical_temporal_accuracy}--\ref{tab:space 1/alp} corroborate our theory. Lastly,
we provide Figure \ref{fig:without exact solution} to show the numerical solution and its cross-section at $x=\frac{1}{2}$ for the second case.
\begin{table}[!h]
  \centering
  \caption{Numerical temporal accuracy for different $\alpha$.}
  \begin{tabularx}{\linewidth}{@{}XXXXXXX@{}}
    \toprule
    \ \ \ \ $\alpha$ & $N$ & \multicolumn{2}{c}{$k=0, \ l = 1$} & \multicolumn{2}{c}{$k=1, \ l = 3$} \ \ \\
    \cline{3-4} \cline{5-6}
    & & $E(M, N, k, l)$ & $\text{rate}_t$  & $E(M,N,k,l)$ & $\text{rate}_t$ \\
    \hline
    \ \ \ \ 0.5 & 200  & 6.6964e-03 & --              & 1.1199e-03 & --       \\
        & 400  & 4.9043e-03 & 0.4493   & 5.5610e-04 & 1.0100   \\
        & 800  & 3.5577e-03 & 0.4631   & 2.7435e-04 & 1.0193   \\
        & 1600 & 2.5627e-03 & 0.4733   & 1.3354e-04 & 1.0387   \\
    \hline
   \ \ \ \ 0.7 & 200  & 1.8911e-03 & --  & 1.1573e-03 & --       \\
        & 400  & 1.1894e-03 & 0.6775   & 5.7071e-04 & 1.0199   \\
        & 800  & 7.3507e-04 & 0.6858   & 2.7978e-04 & 1.0285   \\
        & 1600 & 4.5539e-04 & 0.6908   & 1.3353e-04 & 1.0482   \\
    \hline
  % \ \ \ \ 0.8 & 200  & 8.8971e-04 & --  & 1.2830e-03 & --       \\
    %    & 400  & 5.2456e-04 & 0.7622 & 6.0749e-04 & 1.0276   \\
   %     & 800  & 3.0598e-04 & 0.7777 & 2.9628e-04 & 1.0359   \\
  %      & 1600 & 1.7746e-04 & 0.7859 & 1.4258e-04 & 1.0552   \\
 %   \hline
  \end{tabularx}
  \label{tab:numerical_temporal_accuracy}
\end{table}

\begin{table}[H]
  \centering
  \caption{Numerical spatial accuracy at $t=t_{6N}$ for different $\alpha$.}
    \begin{tabularx}{\linewidth}{@{}XXXXXXX@{}}
    \toprule
   \ \ \ \ {\( M \)} & \multicolumn{2}{c}{\( \alpha = 0.5 \)} & \multicolumn{2}{c}{\( \alpha = 0.6 \)} & \multicolumn{2}{c}{\( \alpha = 0.8 \)}  \\
    \cline{2-3} \cline{4-5} \cline{6-7}\vspace{2pt}
          & \( ||u^{6N}-u_h^{6N}||\) & \( \text{rate}_s \) & \( ||u^{6N}-u_h^{6N}|| \) & \( \text{rate}_s \) & \( ||u^{6N}-u_h^{6N}|| \) & \( \text{rate}_s \)  \\
    \midrule
   \ \ \ \ 8     &  3.6113e-02 & --    & 3.6126e-02 & --    & 3.6155e-02 & --    \\
   \ \ \ \ 16    & 9.0583e-03 & 1.9952 & 9.0600e-03 & 1.9955 & 9.0682e-03 & 1.9953  \\
   \ \ \ \ 32    & 2.2604e-03 & 2.0027 & 2.2593e-03 &  2.0037 & 2.2623e-03 & 2.0030 \\
   \ \ \ \ 64    & 1.6238e-03 & 2.0102 & 5.5694e-04 & 2.0203 & 5.5870e-04 & 2.0176 \\
   \ \ \ \(O(h^2)\) & --    & 2.0000 & --    & 2.0000 & --    & 2.0000  \\
    \bottomrule
    \end{tabularx}\label{tab:spatial2}
  \label{tab:addlabel}
\end{table}

\begin{table}[H]
\centering
\caption{Numerical temporal accuracy with $\alpha=0.5$.}
\begin{tabularx}
{\linewidth}{@{}ccccccccc@{}}
\cline{1-2}\cline{2-3} \cline{4-5} \cline{6-7} \cline{8-9}
$N$ & \multicolumn{2}{c}{$r = \frac{4}{3}$} & \multicolumn{2}{c}{$r = \frac{5}{3}$} & \multicolumn{2}{c}{$r = 2$} & \multicolumn{2}{c}{$r = 3$} \\
\cline{2-9}
& $E(M, N,0,3)$ & $\text{rate}_t$ & $E(M, N, 0, 3)$ & $\text{rate}_t$ & $E(M, N, 0, 3)$ & $\text{rate}_t$ & $E(M, N, 0, 3)$ & $\text{rate}_t$ \\
\cline{1-2}\cline{2-3} \cline{4-5} \cline{6-7} \cline{8-9}
400 & 1.9109e-03 & -- & 7.1956e-04 & -- & 1.5232e-04 & -- & 2.1773e-04 & -- \\
800 & 1.2191e-03 & 0.6484 & 8.8941e-05 & 0.8253 & 7.5500e-05 & 1.0025 & 1.0896e-04 & 0.9987 \\
1600 & 7.7433e-04 & 0.6548 & 2.2845e-04 & 0.8288 & 3.6816e-05 & 1.0162 & 5.3879e-05 & 1.0161 \\
3200 & 4.9047e-04 & 0.6588 & 1.2844e-04 & 0.8307 & 1.7388e-05 & 1.0222 & 2.6092e-05 & 1.0260 \\
$\min\{ar, 1\}$ & -- & 0.6667 & -- & 0.8333 & -- & 1.0000 & -- & 1.0000 \\
\cline{1-2}\cline{2-3} \cline{4-5} \cline{6-7} \cline{8-9}
\end{tabularx}
\label{tab:alp=0.5,r=3/5}
\end{table}

\begin{table}[H]
\centering
\caption{Numerical temporal accuracy with $\alpha=0.7$.}
\begin{tabularx}{\linewidth}{@{}ccccccccc@{}}
\cline{1-2}\cline{2-3} \cline{4-5} \cline{6-7} \cline{8-9}
{$N$} & \multicolumn{2}{c}{$r = \frac{8}{7}$} & \multicolumn{2}{c}{$r = \frac{10}{7}$} & \multicolumn{2}{c}{$r = \frac{15}{7}$} & \multicolumn{2}{c}{$r = 3$} \\
\cline{2-3} \cline{4-5} \cline{6-7} \cline{8-9}
& $E(M, N, 0, 3)$ & $\text{rate}_t$ & $E(M, N, 0, 3)$ & $\text{rate}_t$ & $E(M, N, 0, 3)$ & $\text{rate}_t$ & $E(M, N, 0, 3)$ & $\text{rate}_t$ \\ \cline{1-2}\cline{2-3} \cline{4-5} \cline{6-7} \cline{8-9}
400 & 6.6882e-04
 & -- & 1.1730e-04
 & -- & 1.5841e-04
 & -- &  1.6047e-04 & -- \\
800 &3.8862e-04 &   0.7833
 & 5.9028e-05 & 0.9774 & 8.0820e-05& 0.9709 & 8.2501e-05 & 1.0011 \\
1600 & 2.9033e-05
 &  0.7893
 & 2.9033e-05 & 1.0037 & 4.0512e-05
 & 0.9963 & 4.1613e-05
 & 1.0125 \\
3200 & 1.3696e-05
 & 0.7920 & 1.3696e-05 & 1.0139
 & 1.9742e-05 & 1.0171 & 2.0357e-05
 & 1.0215 \\
$\min\{ar, 1\}$ & -- & 0.8000 & -- & 1.0000 & -- & 1.0000 & -- & 1.0000 \\
\cline{1-2}\cline{2-3} \cline{4-5} \cline{6-7} \cline{8-9}
\end{tabularx}\label{tab:alp=0.7,r=3/5}
\end{table}

\begin{table}[!h]
  \centering
  \caption{Numerical spatial accuracy for different $\alpha$ with $r=\frac{1}{\alpha}$.}
    \begin{tabularx}{\linewidth}{@{}XXXXXXX@{}}
    \toprule
   \ \ \ \ {\( M \)} & \multicolumn{2}{c}{\( \alpha = 0.4 \)} & \multicolumn{2}{c}{\( \alpha = 0.6 \)} & \multicolumn{2}{c}{\( \alpha = 0.8 \)}  \\
   \cline{2-3} \cline{4-5} \cline{6-7}
          & \( E(M, N, 0, 3) \) & \( \text{rate}_s \) & \( E(M, N, 0, 3) \) & \( \text{rate}_s \) & \( E(M, N, 0, 3) \) & \( \text{rate}_s \)  \\
    \midrule
    \ \ \ \ 8     &   1.1474e-01 & --    & 1.2452e-01 & --    & 1.2452e-01 & --    \\
   \ \ \ \ 16    &   2.8761e-02 & 1.9961 & 3.1228e-02 & 1.9954 & 9.0682e-03 & 1.9953  \\
   \ \ \ \ 32    & 7.1492e-03 &  2.0083 &  7.7783e-03  &  2.0053 & 2.2623e-03 & 2.0030 \\
   \ \ \ \ 64    &  1.7390e-03 & 2.0396 & 1.9078e-03 & 2.0275 & 5.5870e-04 & 2.0176 \\
   \ \ \  \(O(h^2)\) & --    & 2.0000 & --    & 2.0000 & --    & 2.0000  \\
    \bottomrule
    \end{tabularx}\label{tab:exact spical grade}
\end{table}

\begin{table}[!h]
  \centering
  \caption{Numerical temporal accuracy for different $\alpha$.}
  \begin{tabularx}{\linewidth}{@{}XXXXXX@{}}
    \hline
   \ \ \ \ $\alpha$ & $N$ & \multicolumn{2}{c}{$k = 0,\ l=1$} & \multicolumn{2}{c}{$k = 1,\ l=3$}  \\
    \cline{3-4} \cline{5-6}
    & & $E(M, N, k , l)$ & $\text{rate}_t$ & $E(M, N, k,l)$ & $\text{rate}_t$  \\
    \hline
   \ \ \ \ 0.6
    & 400  & 3.6320e-03 & --       &  2.1565e-04
        & --           \\
        & 800  & 2.4061e-03 &   0.5940  & 2.4061e-03
               &  0.9826 \\
        & 1600 & 1.5916e-03 &    0.5962
               &  5.5029e-05 &  0.9878    \\
        & 3200 & 1.0519e-03 & 0.5975   & 2.7677e-05 &            0.9915      \\
    \hline
   \ \ \ \ 0.8 & 400  & 8.0498e-04
        & --       & 2.2216e-04 & --       \\
        & 800  & 4.5413e-04 & 0.8259  & 4.5413e-04 &  0.9795
        \\
    & 1600 &2.5813e-04 & 0.8150
      & 5.6914e-05 & 0.9853
        \\
        & 3200 &  1.4740e-04 & 0.8083   & 2.8675e-05 &  0.9890    \\
    \hline
  \end{tabularx}
  \label{tab:The case of uniform grid without true solution}
\end{table}

\begin{table}[H]
  \centering
  \caption{Numerical spatial accuracy at $t=t_{6N}$ for different $\alpha$.}
  \label{tab:spatial_accuracy_N30000}
  \begin{tabularx}{\linewidth}{@{}XXXXXXX@{}}
    \hline
   \ \ \ \ \( M \) & \multicolumn{2}{c}{ \( \alpha = 0.4 \) } & \multicolumn{2}{c}{ \( \alpha = 0.5 \) } & \multicolumn{2}{c}{ \( \alpha = 0.6 \) } \\
    \cline{2-3} \cline{4-5} \cline{6-7}
    & \( ||u^{6N}-u_h^{6N}|| \) & \( \text{rate} \) & \(  ||u^{6N}-u_h^{6N}||\) & \( \text{rate} \) & \(  ||u^{6N}-u_h^{6N}||\) & \( \text{rate} \) \\
    \hline
   \ \ \ \ 8  & 9.2406e--03 & --      & 1.0986e--02 & --      & 9.4463e-03 & --      \\
   \ \ \ \ 16 & 2.3364e--03 & 1.9837  & 2.7790e--03 & 1.9803  & 2.3885e-03 & 1.9837  \\
   \ \ \ \ 32 & 5.5873e--04 & 1.9960  & 6.9678e--04 & 1.9958  & 5.9879e-04 &  1.9960  \\
   \ \ \ \ 64 & 1.4654e--04 & 1.9990  & 1.7432e--04 & 1.9989  & 1.4980e--04 & 1.9990  \\
    \ \ \   \(O(h^2)\)  & --          & 2.0000     & --          & 2.0000     & --          & 2.0000     \\
    \hline
  \end{tabularx}
  \label{tab:space without exact solution}
\end{table}

\begin{table}[H]
\centering
\caption{Numerical temporal accuracy with $\alpha=0.6$.}
\begin{tabularx}{\linewidth}{@{}ccccccccc@{}}
\cline{1-2}\cline{2-3} \cline{4-5} \cline{6-7} \cline{8-9}
{$N$} & \multicolumn{2}{c}{$r = \frac{4}{3}$} & \multicolumn{2}{c}{$r = \frac{5}{3}$} & \multicolumn{2}{c}{$r =2$} & \multicolumn{2}{c}{$r = \frac{7}{3}$} \\
 \cline{2-3} \cline{4-5} \cline{6-7} \cline{8-9}
& $E(M, N,0,3)$ & $\text{rate}_t$ & $E(M, N,0,3)$ & $\text{rate}_t$ & $E(M, N,0,3)$ & $\text{rate}_t$ & $E(M, N,0,3)$ & $\text{rate}_t$ \\ \cline{1-2}\cline{2-3} \cline{4-5} \cline{6-7} \cline{8-9}
200 &  2.1945e-03
 & -- & 9.3063e-04
    & -- &  5.3981e-04 & -- & 6.0398e-04 & -- \\
400 &  1.2739e-03 &  0.7847 & 4.6937e-04 & 0.9875
   & 2.7041e-04    & 0.9973 & 3.0286e-04 & 0.9959\\
800 &  7.3611e-04 & 0.7912 & 2.3571e-04 & 0.9937
    & 1.3541e-04 & 0.9979 & 1.5177e-04   & 0.9967 \\
1600 & 4.2427e-04 & 0.7949 & 1.1811e-04 & 0.9996
    & 6.7779e-05 & 0.9984 & 7.6018e-05 & 0.9975 \\
$\min\{ar, 1\}$ & -- & 0.8000 & -- & 1.0000 & -- & 1.0000 & -- & 1.0000 \\
\cline{1-2}\cline{2-3} \cline{4-5} \cline{6-7} \cline{8-9}
\end{tabularx}\label{tab:The case of symmetric grade grid without true solution}
\end{table}

\begin{table}[H]
  \centering
  \caption{Numerical spatial accuracy for different $\alpha$ with $r=\frac{1}{\alpha}$.}
  \label{tab:space_errors}
  \begin{tabularx}{\linewidth}{@{}ccccccccc@{}}
    \cline{1-2}\cline{2-3} \cline{4-5} \cline{6-7} \cline{8-9}
    \( M \) & \multicolumn{2}{c}{ \( \alpha = 0.4 \) } & \multicolumn{2}{c}{ \( \alpha = 0.5 \) } & \multicolumn{2}{c}{ \( \alpha = 0.6 \) } & \multicolumn{2}{c}{ \( \alpha = 0.7 \) } \\
    \cline{2-3} \cline{4-5} \cline{6-7} \cline{8-9}
    & \( E(M, N,0,3) \) & \( \text{rate}_s \) & \(E(M, N,0,3) \) & \( \text{rate}_s \) & \( E(M, N,0,3)\) & \( \text{rate}_s \) & \( E(M, N,0,3) \) & \( \text{rate}_s \) \\
    \cline{1-2}\cline{2-3} \cline{4-5} \cline{6-7} \cline{8-9}
    8  &  9.2407e-03 & --     & 9.3399e-03 & --     & 9.4463e-03 & --     & 9.5612e--03 & --     \\
    16 & 2.3364e-03 & 1.9837  & 2.3615e-03 & 1.9837  & 2.3885e--03 & 1.9837  & 2.4176e--03 & 1.9836  \\
    32 & 5.8574e--04 & 1.9960  & 5.9204e-04 & 1.9960  & 5.9879e--04 & 1.9960  & 6.0610e--04 & 1.9959  \\
    64 & 1.4654e--04 & 1.9990  & 1.4811e-04 & 1.9990  & 1.4980e--04 & 1.9990  & 1.5163e--04 & 1.9990  \\
     \(O(h^2)\)  & --          & 2.0000  & --          & 2.0000  & --          & 2.0000  & --          & 2.0000  \\
    \cline{1-2}\cline{2-3} \cline{4-5} \cline{6-7} \cline{8-9}
\end{tabularx}\label{tab:space 1/alp}
\end{table}

\begin{figure}[H]
  \centering
  \begin{minipage}{0.50\textwidth}
    \centering
    \includegraphics[width=\textwidth]{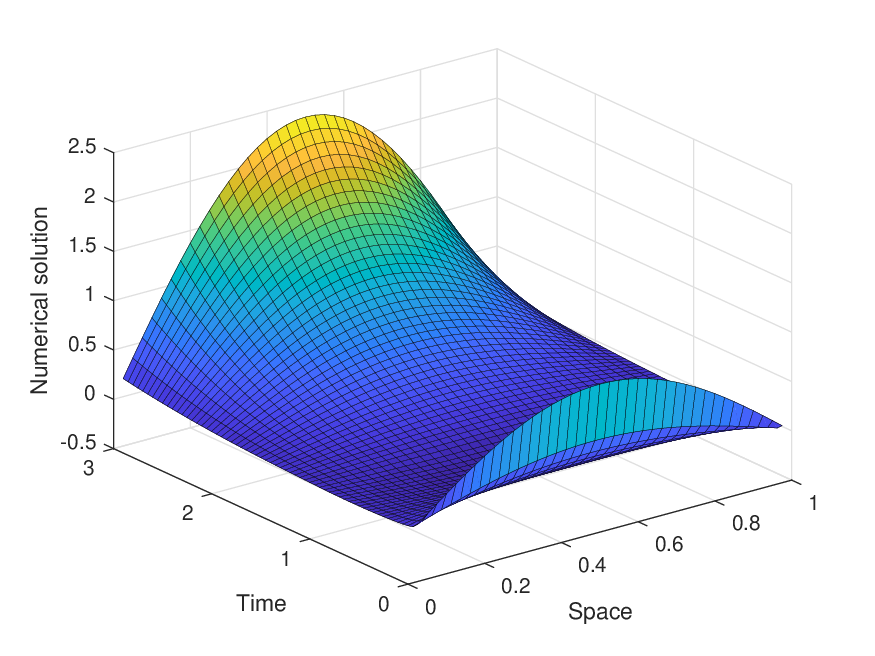} % 替换为第一张图的实际文件名
  \end{minipage}
  \hfill
  \begin{minipage}{0.45\textwidth}
    \centering
    \includegraphics[width=\textwidth]{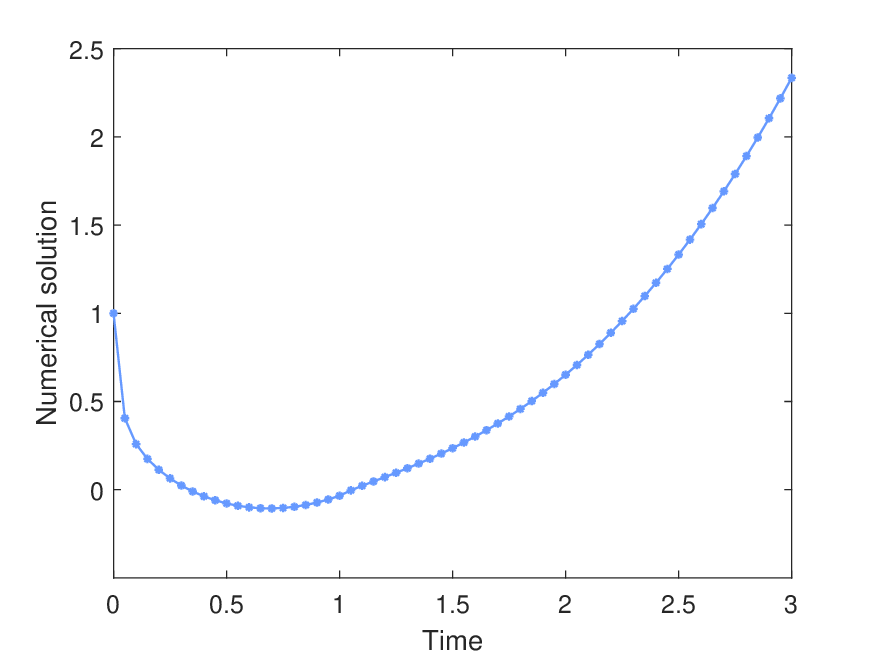} % 替换为第二张图的实际文件名
  \end{minipage}
  \caption{The numerical solution and its cross-section at  $x=\frac{1}{2}$ with $\alpha=0.3,r=1, M=40$ and $N=10.$}
  \label{fig:without exact solution}
\end{figure}

\section{Conclusion}
In this paper, we develop finite element methods for solving a subdiffusion equation with
constant time delay based on both uniform and symmetric graded time meshes. To analyze the
convergence of the proposed numerical scheme, we rigorously examine the local truncation
errors of the L1 formula for Caputo fractional derivative and the fractional right rectangle formula for Riemann-Liouville fractional integral under the assumption of low regularity of
the true solution at $t=0$ and $t=\tau$. For the case of uniform time mesh, we introduce discrete coefficients ${P_k}$ to establish the unconditional stability of the numerical
scheme, subsequently deriving local temporal error estimate. The convergence results
indicate that the maximum temporal error order is $\alpha$ for the interval $(0,\tau]$ and 1 for $(\tau,K\tau].$ We then apply a classical discrete fractional Gronwall inequality to
investigate the stability and global temporal error estimates for the numerical scheme based on the symmetric graded time mesh. The global time convergence order is found to
be $\min\{1,r\alpha\}$. It is worth noting that when $r=1,$ the global convergence order is $\alpha$ in this situation. However, our analysis for the uniform time mesh in Section 3
demonstrates that the maximum temporal error order can reach 1 for the interval
$(\tau,K\tau],$ clearly indicating that the convergence result obtained in Section 3 is
sharp compared with that in Section 4 under this specific situation. Moreover, our
numerical theory suggests that by adjusting the mesh parameter $r$, the accuracy of the
numerical solution can be flexibly tuned between the local and global time convergence
orders. Finally, several numerical experiments are provided to validate the correctness of
the theoretical findings.

\section*{Acknowledgement}
This work is supported by the Research Foundation of Education Commission of Hunan Province of China (Grant No. 23A0126) and the 111 Project (Grant No. D23017).

%\section*{References}

\bibliographystyle{unsrt}

\end{document}